\documentclass[12pt]{amsart}
\usepackage{amssymb}
\usepackage{amsmath}
\title{Harish-Chandra bimodules over quantized symplectic singularities}
\newcommand{\SL}{\operatorname{SL}}
\newcommand{\C}{\mathbb{C}}
\newcommand{\A}{\mathcal{A}}
\newcommand{\gr}{\operatorname{gr}}
\newcommand{\h}{\mathfrak{h}}
\newcommand{\Z}{\mathbb{Z}}
\newcommand{\Orb}{\mathbb{O}}

\newcommand{\B}{\mathcal{B}}
\newcommand{\g}{\mathfrak{g}}

\newcommand{\U}{\mathcal{U}}
\newcommand{\HC}{\operatorname{HC}}
\newcommand{\Hom}{\operatorname{Hom}}

\newcommand{\Str}{\mathcal{O}}
\newcommand{\M}{\mathcal{M}}
\newcommand{\Q}{\mathbb{Q}}
\newcommand{\D}{\mathcal{D}}
\newcommand{\Weyl}{\mathbb{A}}
\newcommand{\VA}{\operatorname{VA}}
\newcommand{\jet}{\operatorname{J}^\infty}
\newcommand{\Irr}{\operatorname{Irr}}
\newcommand{\R}{\mathbb{R}}
\newcommand{\Aut}{\operatorname{Aut}}
\newcommand{\Der}{\operatorname{Der}}
\newcommand{\bA}{\bar{\bf{A}}}
\newcommand{\J}{\mathcal{J}}
\newcommand{\Walg}{\mathcal{W}}
\newcommand{\I}{\mathcal{I}}
\newcommand{\Disk}{\mathbb{D}}

\newtheorem{Thm}{Theorem}[section]
\newtheorem{Prop}[Thm]{Proposition}
\newtheorem{Cor}[Thm]{Corollary}
\newtheorem{Lem}[Thm]{Lemma}
\theoremstyle{definition}
\newtheorem{Ex}[Thm]{Example}

\newtheorem{Rem}[Thm]{Remark}
\newtheorem{Conj}[Thm]{Conjecture}
\numberwithin{equation}{section}
\author{Ivan Losev}
\address{Department
of Mathematics, Yale University, CT, USA}
\email{ivan.loseu@gmail.com}
\begin{document}
\begin{abstract}
In this paper we classify the irreducible Harish-Chandra bimodules with full support
over filtered quantizations of conical symplectic singularities under the condition
that none of the slices to codimension 2 symplectic leaves has type $E_8$. More precisely,
consider the quantization $\A_\lambda$ with parameter $\lambda$.
We show that the top quotient $\overline{\HC}(\A_\lambda)$ of the category of
Harish-Chandra $\A_\lambda$-bimodules
embeds into the category of representations of the algebraic fundamental group,
$\Gamma$, of the open leaf. The image coincides with the representations of
$\Gamma/\Gamma_\lambda$, where $\Gamma_\lambda$ is a normal subgroup of
$\Gamma$ that can be recovered from the quantization parameter $\lambda$ combinatorially.
As an application of
our results, we describe the Lusztig quotient group in terms of the geometry
of the normalization of the orbit closure in almost all cases.
\end{abstract}
\maketitle

\begin{center}{\it To the memory of Ernest Borisovich Vinberg.}\end{center}

\section{Introduction}
\subsection{Harish-Chandra bimodules over quantizations of symplectic singularities}\label{SS_intro_HC}
The goal of this paper is to study Harish-Chandra bimodules over quantizations of conical symplectic singularities.

Let us start by defining Harish-Chandra bimodules in the general setting
of filtered quantizations of graded Poisson algebras.

Let $A$ be a finitely generated commutative associative unital algebra. Suppose that
$A$ is equipped with two additional structures: an algebra grading $A=\bigoplus_{i=0}^\infty A_i$
such that $A_0=\C$ and a Poisson bracket $\{\cdot,\cdot\}$ of degree $-d$,
where $d\in \Z_{>0}$, which, by definition, means that $\{A_i,A_j\}\subset A_{i+j-d}$
for all $i,j$.  By a {\it filtered quantization} of $A$ we mean a pair $(\A,\iota)$ of
\begin{itemize}
\item a filtered associative algebra
$\A=\bigcup_{i\geqslant 0}\A_{\leqslant i}$ such that $[\A_{\leqslant i},\A_{\leqslant j}]
\subset \A_{\leqslant i+j-d}$,
\item a graded Poisson algebra
isomorphism  $\iota:\gr\A\xrightarrow{\sim} A$.
\end{itemize}

 Following \cite[Section 2.5]{HC}, by a Harish-Chandra (shortly, HC)
$\A$-bimodule we mean an $\A$-bimodule $\B$ that can be equipped with an increasing
exhaustive bimodule filtration $\B=\bigcup_{j}\B_{\leqslant j}$ such that $[\A_{\leqslant i},\B_{\leqslant j}]
\subset \B_{\leqslant i+j-d}$ (which implies that the actions of $A$ on $\gr\B$
from the left and from the right coincide) and $\gr\B$ is a finitely generated $A$-module. Such a filtration
will be called {\it good}. For example, the regular bimodule $\A$ is HC.

The most classical example here is when $\A=U(\g)$ for a semi-simple Lie algebra $\g$,
here $d=1$. A Harish-Chandra bimodule is the same thing as a finitely generated $U(\g)$-bimodule with locally finite
adjoint action of $\g$. These bimodules are extensively studied in  Lie representation theory.

The algebras $\A$ we are interested in are filtered quantizations of conical symplectic
singularities.

Let us recall the definition of a conical symplectic singularity.
Let $Y$ be a normal Poisson algebraic variety such that the Poisson bracket on the smooth
locus $Y^{reg}$ is non-degenerate. Let $\omega$ denote the corresponding symplectic form on
$Y^{reg}$. Following Beauville, \cite{Beauville}, we say that $Y$ has symplectic singularities if
there is a resolution of singularities $\rho:\tilde{Y}\rightarrow Y$ such that
$\rho^*\omega$ extends to a regular 2-form on $\tilde{Y}$. We say that a variety $Y$
with symplectic singularities is a {\it conical symplectic singularity} if it is equipped
with an action of the one-dimensional torus $\C^\times$  such that
\begin{itemize}
\item
$\C^\times$ contracts $Y$ to a single point (and so $Y$ is automatically affine),
\item and the degree of the Poisson bracket on $\C[Y]$ is $-d$ for $d\in \Z_{>0}$.
\end{itemize}
In particular, $A:=\C[Y]$ is a graded Poisson algebra as above.

Examples of conical
symplectic singularities include the following.
\begin{enumerate}
\item The nilpotent cone $\mathcal{N}$ in a semisimple Lie algebra $\g$. More generally,
let $\Orb\subset \g$ be a nilpotent orbit and $\tilde{\Orb}$ be its $G$-equivariant cover
(where $G$ stands for the simply connected group with Lie algebra $\g$). Then the algebra
$\C[\tilde{\Orb}]$ is finitely generated and $Y:=\operatorname{Spec}(\C[\tilde{\Orb}])$
is a conical symplectic singularity.  For $\tilde{\Orb}=\Orb$ this was observed in
\cite{Beauville} based on prior results of Panyushev. The general case follows from there,
see Lemma \ref{Lem:symp_sing_cover}.
\item Let $V$ be a symplectic vector space and $\Gamma$ be a finite group of linear
symplectomorphisms of $V$. Then $Y:=V/\Gamma$ is a conical symplectic singularity,
\cite{Beauville}.
\end{enumerate}
There are many other examples  of conical symplectic singularities (preimages of Slodowy slices in $\operatorname{Spec}(\C[\tilde{\Orb}])$,
affine Nakajima and hypertoric varieties, etc.) but only (1) and (2) are important for the present
paper.

%Now let us explain what we mean by a filtered quantization of a conical symplectic singularity $Y$.
%By definition, this is a filtered associative algebra $\A=\bigcup_{i=0}^\infty \A_{\leqslant i}$
%satisfying $[\A_{\leqslant i},\A_{\leqslant j}]\subset \A_{\leqslant i+j-d}$ together with a
%graded Poisson algebra isomorphism $\gr\A\cong \C[Y]$.

For a general conical symplectic singularity $Y$  the filtered quantizations of $Y$
(i.e., of $\C[Y]$) were classified in \cite{orbit}. The result can be stated as follows -- we will recall
it in more detail below in Section \ref{SS_HC_symp_sing}. There is a finite dimensional $\C$-vector space $\h_Y^*$ defined
over $\Q$ and a finite crystallographic reflection group $W_Y$ acting on $\h_Y^*$ such
that the filtered quantizations of $Y$ are in a natural one-to-one correspondence with
$\h_Y^*/W_Y$. We will write $\A_\lambda$ for the filtered quantization corresponding
to $\lambda\in \h_Y^*$.

In the examples of conical symplectic singularities
mentioned above we get algebras of great interest for Representation theory.
When $Y$ is the nilpotent cone in $\g$, its filtered quantizations are the central reductions
of $U(\g)$, while for $Y:=\operatorname{Spec}(\C[\tilde{\Orb}])$ we get interesting Dixmier
algebras in the sense of Vogan, \cite{Vogan_Dixmier}. In the case when $Y=V/\Gamma$ we get  spherical symplectic reflection algebras of Etingof and Ginzburg, \cite{EG}.

The goal of this paper is to classify irreducible Harish-Chandra $\A_\lambda$-bimodules that are faithful
as left or, equivalently, right modules (both conditions are equivalent to the condition that the associated variety of
the bimodule  coincides with $Y$, we will recall associated varieties in Section
\ref{SS_Poisson_HC}). Below we will say that such HC bimodules
are of {\it full support}. One could hope that the classification in this case will shed
some light on that of HC bimodules with arbitrary associated varieties. We note that
the opposite case to what we consider is of finite dimensional irreducible bimodules.
Here the classification reduces to that of finite dimensional irreducible modules
and very much depends on the algebra. In contrast, the classification of irreducible
HC bimodules with full support is geometric, as we will see below. One should expect that the case of
general associated varieties interpolates between the two extreme cases.

The classification  of irreducible HC bimodules with full support is known in many cases. For example,
the result for $Y=\mathcal{N}$ is classical -- it will be recalled below in Section
\ref{SS_classif_1dim} -- as this case turns
out to be important for the classification in the general case. Various special
cases and partial results  for $Y=\operatorname{Spec}(\C[\Orb])$ were obtained
in \cite{HC,LO,quant_nilp}. The case of $Y=V/\Gamma$ was considered in \cite{sraco}
and then in \cite{S}. In the latter paper a complete classification was obtained in the case when
$V=U\oplus U^*$ and $\Gamma$ acts on $U$ as a complex reflection group.

However, even in some simple cases, most notably for the Kleinian singularities
$V/\Gamma$, where $\dim V=2$ and $\Gamma$ is not cyclic,
the classification is not known. It turns out that this case is crucial for understanding the case
of general $Y$. We consider the Kleinian case in the next section.

\subsection{Results for quantizations of Kleinian singularities}\label{SS_Klein_HC}
So let $Y=\C^2/\Gamma$. Recall that, up to conjugation in $\operatorname{SL}_2(\C)$,
the subgroups $\Gamma$ are classified by the type ADE Dynkin diagrams. In particular,
to $\Gamma$ we can assign the  Cartan space $\h_\Gamma$ and
the Weyl group $W_\Gamma$ (of the corresponding ADE type). We have $\h_Y^*=\h_\Gamma^*,
W_Y=W_\Gamma$.

Quantizations of $Y$ were extensively studied in the past with various
constructions given in \cite{CBH} (a special case of the general symplectic reflection algebra
construction), \cite{Holland} (as a quantum Hamiltonian reduction), \cite{Premet}
(as the central reduction of a suitable finite W-algebra). All these constructions
give the same quantizations, see \cite[Theorems 5.3.1,6.2.2]{quant}.

Let us describe the classification of irreducible HC bimodules with full support
in this case. We start with a discussion of quantization parameters, see
Example \ref{Ex:Kleinian_quantizations} for details.  We have an affine isomorphism
between $\h_\Gamma^*$ and the affine subspace  $(\C\Gamma)^\Gamma_1\subset (\C\Gamma)^\Gamma$ consisting of elements
$c$ of the form $1+\sum_{\gamma\neq 1}c_\gamma \gamma$. Namely to
$c\in (\C\Gamma)_1^\Gamma$ we assign $\lambda_c\in \h_Y^*$ with $\langle \lambda_c,\alpha_i^\vee\rangle=
\operatorname{tr}_{N_i}(c)$. Here we write $\alpha_i^\vee$ for a simple coroot in $\h_\Gamma$
and $N_i$ for the corresponding nontrivial irreducible representation of $\Gamma$.

Now we proceed to a conjectural classification result for  irreducible HC $\A_\lambda$-bimodules
with full support. Consider the affine Weyl group
$W^a_\Gamma:=W_\Gamma\ltimes \Lambda_r$, where $\Lambda_r$ is the root lattice in $\h_\Gamma^*$.
The group $W^a_\Gamma$ naturally acts on $\h_\Gamma^*$ by affine transformations.

\begin{Conj}\label{Conj:Klein_classif_prelim}
The following claims are true:
\begin{enumerate}
\item For each $c\in (\C\Gamma)_1^\Gamma$, there is a minimal normal subgroup $\Gamma_c\subset \Gamma$ such that $W^a_\Gamma \lambda_c$ contains $\lambda_{c'}$ with $c'\in \C\Gamma_c(\subset \C\Gamma)$.
\item The irreducible  HC $\A_{\lambda_c}$-bimodules with full support are in bijection with the irreducible
representations of $\Gamma/\Gamma_c$.
\end{enumerate}
\end{Conj}

\begin{Thm}\label{Thm:Klein_classif_prelim}
Conjecture \ref{Conj:Klein_classif_prelim} is true when $\Gamma$ is not of type $E_8$.
\end{Thm}

Under the same restriction, we can also describe the top quotient of the category of
HC bimodules as a tensor category, Theorem \ref{Thm:Klein_classif}.

At this point, we do not know what happens in the $E_8$-case. Clearly, (1) should not be
difficult to check. On the other hand, the conclusion of (2) is true for some normal
subgroup of $\Gamma$, we just do not know which of the three normal subgroups to take.
What makes type $E_8$ special is that the corresponding Kleinian group is not solvable.

We also note that the type A case ($=$ cyclic $\Gamma$) of Conjecture \ref{Conj:Klein_classif_prelim}
 was proved in \cite{S}.

The most essential ingredient of the proof of Theorem \ref{Thm:Klein_classif_prelim} is to relate
the HC bimodules over $\A_\lambda$ and over the central reduction $\U_\lambda$ of $U(\g)$ (where $\g$ is a semisimple
Lie algebra of the same type as $\Gamma$) corresponding to $\lambda$. This is a special case of the general extension result for HC bimodules that also allows to reduce the classification in general to that in the Kleinian case.

\subsection{Results for quantizations of symplectic singularities}\label{SS_HC_symp_sing}
Now assume that $Y$ is a general conical symplectic singularity. Pick $\lambda\in \h_Y^*$ and let $\A_\lambda$
  be the corresponding  filtered quantization. Consider the algebraic
fundamental group $\Gamma$ of $Y^{reg}$. Recall that this group is the pro-finite
completion of $\pi_1(Y^{reg})$. The finite index subgroups of $\Gamma$ are in one-to-one
correspondence with finite etale covers of $Y^{reg}$. By a result of Namikawa,
\cite{Namikawa_fund}, $\Gamma$ is a finite group. Any finite dimensional representation
of $\pi_1(Y^{reg})$ factors through $\Gamma$.

More precisely, we will see that $\lambda$ defines a normal
subgroup of $\Gamma$, to be denoted by $\Gamma_\lambda$, and the set of irreducible
HC bimodules with full support is in a bijection with $\Irr(\Gamma/\Gamma_\lambda)$,
the set of isomorphism classes of irreducible representations of $\Gamma/\Gamma_\lambda$.

Let us explain how to construct the normal subgroup $\Gamma_\lambda$. By a result of
Kaledin, \cite{Kaledin}, $Y$ has finitely many symplectic leaves. Let $\mathcal{L}_1,\ldots,
\mathcal{L}_k$ be the codimension 2 leaves. Let $\Sigma_1,\ldots,\Sigma_k$
be formal slices through $\mathcal{L}_1,\ldots,\mathcal{L}_k$. Then $\Sigma_i=\Disk^2/\Gamma_i$,
where we write $\Disk^2$ for $\operatorname{Spec}(\C[[x,y]])$.
So we can consider the corresponding Cartan space $\tilde{\h}_i^*$ for $\Gamma_i$.
The fundamental group $\pi_1(\mathcal{L}_i)$ acts on $\tilde{\h}_i^*$ by monodromy.
Let $\h_i^*:=(\tilde{\h}_i^*)^{\pi_1(\mathcal{L}_i)}$. Then we have $\h_Y^*=\bigoplus_{i=0}^k \h_i^*$,
where $\h_0^*:=H^2(Y^{reg},\C)$, see \cite[Lemma 2.8]{orbit}.

Let us write $\lambda_i$ for the component of $\lambda$ in $\h_i^*\subset \tilde{\h}_i^*$ and let
$c_i\in \C\Gamma_i$ be the element corresponding to $\lambda_i$ as explained in the previous section.
Let us write $\Gamma_{i,\lambda}$ for the normal subgroup $\Gamma_{i,c_i}$ from Conjecture
\ref{Conj:Klein_classif_prelim}.

Now note that we have a natural group homomorphism $\Gamma_i=\pi_1^{alg}(\Sigma_i\setminus \{0\})\rightarrow \Gamma$ induced by the inclusion $\Sigma_i\setminus \{0\}\hookrightarrow Y^{reg}$. Let $\Gamma_\lambda$ be the minimal
normal subgroup of $\Gamma$ containing the images of  all $\Gamma_{i,\lambda}$.

\begin{Thm}\label{Thm:classif_sympl_prelim}
Suppose that Conjecture \ref{Conj:Klein_classif_prelim} holds for all $\Gamma_i, i=1,\ldots,k$.
Then there is a bijection between the  irreducible HC $\A_\lambda$-bimodules
with full support and the irreducible $\Gamma/\Gamma_\lambda$-modules.
\end{Thm}

There is a stronger version on the level of tensor categories, Theorem \ref{Thm:HC_classif}.

So we have a full classification of irreducible HC $\A_\lambda$-bimodules with full support
in the case when $Y$ has no two-dimensional slices of type $E_8$. This is the case in
the majority of interesting examples. For instance, if we are dealing
with $Y=\operatorname{Spec}(\C[\tilde{\Orb}])$, see Example (2) in Section \ref{SS_intro_HC},
then this assumption fails precisely when $Y$ is the nilpotent cone in the Lie algebra of
type $E_8$. Of course, in that case the classification is also known.

\begin{Rem}\label{Rem:HC_diff_param}
More generally, for two different filtered quantizations $\A_{\lambda'}$, $\A_\lambda$ of $Y$ one can consider
HC $\A_{\lambda'}$-$\A_\lambda$-bimodules and ask to classify such  irreducible bimodules
with full support. The situation there is more complicated than in the case of $\lambda'=\lambda$,
we have partial results on the classification (including a complete classification  when $\h_0^*=0$) that we will explain  sketching required modifications in Section \ref{SS_classif_diff_param}. One reason we chose to omit
complete proofs in the general case is that the case $\lambda'= \lambda$ is much less technical
but also is more important for applications, including those in Lie representation theory.
\end{Rem}

\subsection{Applications and variants}
One application of Theorem \ref{Thm:classif_sympl_prelim} is  a geometric interpretation
of Lusztig's quotients, \cite[Section 13]{orange}, for almost all cases (with the exception of
one case in $E_7$ and three in $E_8$). These finite groups were introduced by Lusztig in
his work on computing the characters for finite groups of Lie type. Namely, from a
two-sided cell $c$ in a Weyl group $W$ Lusztig has produced a finite group $\bA_c$.
He also established a connection of this group to nilpotent orbits, as follows.

Let $\g$ be a semisimple Lie algebra with Weyl group $W$. Then the two-sided cells in $W$
are in one-to-one correspondence with the so called {\it special} orbits  in $\g$.
Let $\Orb_c$ denote the orbit corresponding to $c$.
Lusztig has proved that $\bA_c$ can be realized as a quotient of the component
group ${\bf A}(\Orb_c)$, that is the $G$-equivariant fundamental group of
$\Orb_c$, where $G:=\operatorname{Ad}(\g)$.

The quotients $\bA_c$ were further studied in a number of papers including \cite{LO}. There the author and
Ostrik computed $\bA_c$ in terms of the two-sided $W$-module $[c]$ corresponding
to $c$ and the Springer representation of $W\times {\bf A}(\Orb_c)$ associated to $\Orb_c$. Using this,
we have identified the semi-simple part of the subquotient of $\HC(\U_\rho)$ corresponding to $\Orb_c$
(this subquotient categorifies $[c]$) with the category $\mathsf{Sh}^{\bA_c}(Y_c\times Y_c)$,
where $Y_c$ is the category of finite dimensional modules over the W-algebra corresponding
to $\Orb_c$. Below in Section \ref{S_Lusztig_quot} we will use this result from \cite{LO} and Theorem
\ref{Thm:classif_sympl_prelim} to show that $\bA_c=\Gamma/\Gamma_\lambda$, where
$\Gamma=\pi_1(\Orb_c)$ and $\lambda$ is suitable quantization parameter
for $\C[\Orb_c]$ (that exists for all $\Orb_c$ but the four mentioned above).
This gives a new description of $\bA_c$ basically in terms of the
geometry of $\operatorname{Spec}(\C[\Orb_c])$. The main results of Section
\ref{S_Lusztig_quot} are Propositions \ref{Prop:integral_period} and
\ref{Prop:Lusztig_quotient_comput}.

We now mention some subsequent work. In \cite{LMBM} we give a new definition of
unipotent Harish-Chandra bimodules over semisimple Lie algebras and apply Theorem \ref{Thm:classif_sympl_prelim}
to classify and study them.
And in \cite{LY}, we prove an analog of Theorem \ref{Thm:classif_sympl_prelim} for irreducible
Harish-Chandra modules over quantizations of $\C[\Orb]$, where $\Orb$ is a nilpotent
orbit in a semisimple Lie algebra satisfying $\operatorname{codim}_{\overline{\Orb}}\partial\Orb\geqslant 4$.

{\bf Acknowledgements}. I would like to thank Pavel Etingof, George Lusztig,
Dmytro Matvieievskyi and Victor Ostrik for stimulating discussions. I would also like
to thank Dmytro Matvieievskyi and Shilin Yu for the many comments that allowed me to improve the
exposition.  This work has been funded by the  Russian Academic Excellence Project '5-100'.
This work was also partially supported by the NSF under grant DMS-1501558. This paper is
dedicated to the memory of my advisor, Ernest Borisovich Vinberg, who sadly passed away in
May 2020.

\section{Preliminaries}\label{S_prelim}
\subsection{Non-commutative period map}\label{SS_Per_NC}
In this section we will discuss quantizations of smooth symplectic algebraic varieties and their
important invariant, the non-commutative period,  following \cite{BK_quant,quant}.

Let $X$ be a symplectic algebraic variety. So $\Str_X$ is a Poisson sheaf of algebras.
By a formal quantization of $X$ we mean a pair  $(\D_h,\iota)$, where
\begin{itemize}
\item $\D_h$ is a sheaf in Zariski topology of $\C[[h]]$-algebras on $X$
that is $\C[[h]]$-flat, and complete and separated in the $h$-adic topology,
\item and $\iota:\D_h/h\D_h\xrightarrow{\sim} \Str_X$ is an isomorphism
of sheaves of Poisson algebras on $X$.
\end{itemize}
We note that in the case when $X$ is affine, to give a formal quantization of $X$
is the same as to give a formal quantization of $\C[X]$.

Bezrukavnikov and Kaledin in \cite[Section 4]{BK_quant} defined an invariant of $\D_h$
called the {\it non-commutative period} that lies in $H^2(X,\C[[h]])$. Let us explain the construction
as we will need it below.

The first step in the construction is passing from a quantization $\D_h$ to its quantum jet bundle,
to be denoted by $\jet \D_h$, that is a pro-coherent sheaf of $\C[[h]]$-algebras on $X$ equipped with a flat connection.

Let us start with recalling the usual jet bundle $\jet \Str_X$. Consider $X\times X$ with the projections $p_1,p_2:X\times X
\rightarrow X$. By the jet bundle $\jet\Str_X$ we mean $p_{1*}(\widehat{\Str}_\Delta)$,
where we write $\widehat{\Str}_\Delta$ for the completion of $\Str_{X\times X}$ along the
diagonal $\Delta$. This is a pro-coherent sheaf on $\Str_X$ whose fiber
at $x\in X$ is the completion $\Str_X^{\wedge_x}$ at $x$.
This bundle comes with a flat connection $\nabla$ (derivatives along
the first copy of $X$). The subsheaf of flat sections $\left(\jet\Str_X\right)^\nabla$
is identified with $\Str_X$ via $p_2^*$. Finally, note that $\jet \Str_X$
comes with a natural $\Str_X$-linear Poisson structure.

Now let $\D_h$ be a formal quantization of $\Str_X$. Then we can form the quantum
jet bundle $\jet\D_h$: we consider the completion of $\Str_X\otimes \D_h$ along
the diagonal $\Delta$, denote this sheaf by $\widehat{\D}_{h,\Delta}$. Then $\jet\D_h:=
p_{1*}\widehat{\D}_{\hbar,\Delta}$. Again, this is a pro-coherent sheaf on $X$ with a
flat connection. The sheaf of flat sections of this connection is $\D_h$ and $\jet\D_{h}/(h)=\jet\Str_X$.

Let $\Weyl_h$ denote the formal Weyl algebra in $\dim X$-variables, the unique formal
quantization of the Poisson algebra $\C[[x_1,\ldots,x_n,y_1,\ldots,y_n]]$ (with the standard
Poisson bracket), where $\dim X=2n$. The sheaf $\jet \D_h$ defines a  torsor over the Harish-Chandra pair
$(\Aut \Weyl_h, \Der \Weyl_h)$. The sheaf $\jet\D_h$ is the associated
bundle of this torsor with fiber $\Weyl_h$.
The assignment sending $\D_h$ to that torsor is a bijection between
\begin{itemize}
\item
the set of isomorphism
classes of quantizations,
\item and the set of isomorphism classes of Harish-Chandra torsors
over $(\Aut \Weyl_h, \Der \Weyl_h)$ that specialize to the torsor of formal coordinate
systems at $h=0$.
\end{itemize}

The map $h^{-1}a\mapsto h^{-1}[a,\cdot]$ is an epimorphism $h^{-1}\Weyl_\hbar\twoheadrightarrow
\Der \Weyl_\hbar$ with kernel $h^{-1}\C[[h]]$. The exact sequence of Lie algebras
$$0\rightarrow h^{-1}\C[[h]]\rightarrow h^{-1}\Weyl_h\rightarrow \Der \Weyl_h\rightarrow 0$$
lifts to an exact sequence of Harish-Chandra pairs
$$0\rightarrow (h^{-1}\C[[h]], h^{-1}\C[[h]])\rightarrow \mathsf{G}\rightarrow
(\Aut\Weyl_h, \Der \Weyl_h)\rightarrow 0.$$
Here the first torsor corresponds to the additive group $h^{-1}\C[[h]]$
and $\mathsf{G}$ is defined in \cite[Section 3.2]{BK_quant}.
The exact sequence gives rise to the map $\mathsf{Per}: \mathsf{Quant}(X)
\rightarrow H^2_{DR}(X,\C[[h]])$, where we write
$\mathsf{Quant}(X)$ for the set of isomorphism classes of formal quantizations
of $X$. This map sends a quantization $\D_h$ to the obstruction class for
lifting the corresponding Harish-Chandra torsor to a $\mathsf{G}$-torsor.
The degree $0$ term of $\mathsf{Per}(\D_h)$ is the class of the symplectic
form $\omega$ on $X$. By the construction, $\mathsf{Per}(\D_h)\operatorname{mod}h^2$
is recovered from $\D_h/(h^2)$ together with a ``non-commutative Poisson bracket''
$\{\cdot,\cdot\}$ induced by the Lie bracket on $\D_h/(h^3)$.

We will be interested in the situation when $\C^\times$ acts
on $X$ with $t.\omega=t^d \omega$ for $d\in \Z_{>0}$.
Of course, here the cohomology class of $\omega$ is $0$. We say that a formal quantization $\D_h$ is {\it graded}
if the action of $\C^\times$ on $\Str_X$ lifts to an action of $\C^\times$
on $\D_h$ by $\C$-algebra automorphisms such that $t.h=t^d h$ for $t\in \C^\times$
and $\iota:\D_h/h\D_h\xrightarrow{\sim}\Str_X$ is $\C^\times$-equivariant. It was shown in \cite[Section 2.3]{quant}
that if $\D_h$ is graded, then $\mathsf{Per}(\D_h)\in h H^2_{DR}(X)$.

The construction of the period generalizes to the relative situation,
\cite[Section 4]{BK_quant}.
Let $S$ be a  scheme over $\C$. We will mostly be interested in the case
when $S=\operatorname{Spec}(\C[t]/(t^2))$. Let $X$ be a smooth symplectic scheme (of finite type) over
$S$ (meaning, in particular, that now $\omega\in \Omega^2(X/S)$), let $\pi:X\rightarrow S$
be the corresponding morphism. The notion of
a formal quantization still makes sense but now $\D_h$ is a sheaf of  $\pi^{-1}\Str_S[[h]]$-algebras
and $\iota$ is $\pi^{-1}\Str_S$-linear. Here we write $\pi^{-1}$ for the sheaf-theoretic pullback.
The set of isomorphism classes of the formal quantizations
of $X$ will be denoted by $\mathsf{Quant}(X/S)$. To $\D_h$ we can assign its period $\mathsf{Per}(\D_h)\in
H^2_{DR}(X/S)[[h]]$ in the same way as before.

%Another important property of the period is that it is functorial under \'{e}tale morphisms.

We will need to understand the behavior of the period under regluing.
Namely, let us take a graded formal quantization $\D_h$. Cover $X$
with $\C^\times$-stable open affine subsets $U_i$ and let us write
$U_{ij}$ for $U_{i}\cap U_j$. Let us pick a 1-cocycle $\theta=(\theta_{ij})$
of $\C^\times$-equivariant $\C[[h]]$-linear automorphisms of $\D_h|_{U_{ij}}$. In particular, $\theta_{ji}=\theta_{ij}^{-1}$
and we have the equality $\theta_{ik}=\theta_{ij}\theta_{jk}$ of automorphisms
of $\D_{h}|_{U_{ijk}}$. We can form a new quantization $\D_h^\theta$ obtained from
$\D_h$ by twisting with $\theta$. We want to relate the periods
$\mathsf{Per}(\D_h)$ and $\mathsf{Per}(\D_h^\theta)$.

Note that $\theta_{ij}=\exp(h\delta_{ij})$, where $\delta_{ij}$ is a derivation
$\D_h|_{U_{ij}}$ of degree $-d$. Let $\delta^0_{ij}$ denote
$\delta_{ij}$  modulo $h$. This is a symplectic vector field on $U_{ij}$
of degree $-d$. Let $\alpha_{ij}$ be the corresponding 1-form (obtained
by pairing $\delta_{ij}$ and $\omega$). Note that $\alpha_{ij}$ is closed
and has degree $0$. The forms $\alpha_{ij}$  form a \v{C}ech
and hence a \v{C}ech-De Rham cocycle. Let $[\alpha]$ denote its class
in  $H^2_{DR}(X)$.

\begin{Lem}\label{Lem:period_regluing}
We have $\mathsf{Per}(\D_h^\theta)=\mathsf{Per}(\D_h)+h[\alpha]$.
\end{Lem}
\begin{proof}
Since both quantizations are graded, we have $\mathsf{Per}(\D_h^\theta), \mathsf{Per}(\D_h)\in h H^2_{DR}(X)$. It remains to show that  we have $\mathsf{Per}(\D_h^\theta)=\mathsf{Per}(\D_h)+h[\alpha]$
modulo $h^2$. For this, consider  the scheme $X\times S$ over $S$, where $S:=\operatorname{Spec}(\C[t]/(t^2))$, and its quantization $\D_h\otimes \C[t]/(t^2)$.
We can twist the sheaf $\D_h\otimes \C[t]/(t^2)$ with the cocycle $1+t\delta_{ij}$, denote
the result by $(\D_h\otimes \C[t]/(t^2))^\theta$. This is a quantization of the corresponding
twist $(X\times S)^\theta$. The class of the fiberwise symplectic form is $t[\alpha]$.
Now consider the specialization of $(\D_h\otimes \C[t]/(t^2))^\theta$  to $th$. We get
the sheaf of algebras over $\C[h]/(h^2)$ that comes with the bracket $\{\cdot,\cdot\}$
induced from the Lie bracket on $(\D_h\otimes \C[t]/(t^2))^\theta$. We have an isomorphism of
this specialization with $D_h^\theta/(h^2)$ that is compatible with the brackets.
We conclude that the period of $D_h^\theta$ mod $h^2$ coincides with the specialization
of that of  $(\D_h\otimes \C[t]/(t^2))^\theta$ to $th$ (where we then need to change the
variable $th$ back to $h$). The latter is $\mathsf{Per}(\D_h)+th[\alpha]$. This is equivalent to the formula in the statement of the lemma.
\end{proof}

\subsection{Classification of quantizations of symplectic varieties}\label{SS_quant_classif_sympl}
Let us now discuss classification questions and some consequences.

The next claim follows from \cite[Theorem 1.8]{BK_quant}.

\begin{Prop}\label{Prop:BK_main}
Let $S$ be a $\C$-scheme of finite type and  $X$ be a smooth symplectic $S$-scheme of finite type. Assume that
$H^i(X,\Str_X)=0$ for $i=1,2$. Then the map $\mathsf{Quant}(X/S)
\rightarrow [\omega]+h H^2_{DR}(X/S)[[h]]$ is a bijection.
\end{Prop}

This proposition has the following corollary proved in \cite[Section 2.3]{quant}.

\begin{Cor}\label{Cor:graded_filt_quant}
Let $S=\operatorname{pt}$, $X$ be as in Proposition \ref{Prop:BK_main},
and we have a $\C^\times$-action on $X$ as before.
Then the period map gives a bijection between the isomorphism classes of
graded formal quantizations and $h H^2(X,\C)$.
\end{Cor}

We are going to use Proposition \ref{Prop:BK_main} to study the derivations
of quantizations of affine varieties.

\begin{Lem}\label{Lem:deriv_lift}
Let $X$ be an affine smooth symplectic variety and $\delta_0$ be a Poisson derivation
of $\C[X]$. Let $\D_h$ be a formal quantization of $X$. Then $\delta_0$ lifts
to a derivation of the $\C[[h]]$-algebra $\D_h$.
\end{Lem}
\begin{proof}
Set $S:=\operatorname{Spec}(\C[t]/(t^2))$ and $\tilde{X}:=X\times S$. Thanks to the Gauss-Manin
connection, we have an identification $H^2_{DR}(\tilde{X}/S)=H^2_{DR}(X)\times (\C[t]/(t^2))$.
We can consider two quantizations of $\tilde{X}$. First, we have $\tilde{\D}^1_h:=\D_h\otimes \C[t]/(t^2)$.
Next, we have an automorphism $1+t\delta_0$ of $\C[\tilde{X}]$. Let $\tilde{\D}^2_h$ be the twist of
$\tilde{\D}^1_h$ under this automorphism. Since $1+t\delta_0$ acts trivially on the De Rham cohomology,
we see that the periods of $\tilde{\D}^1_h$ and $\tilde{\D}^2_h$ are the same.
Therefore, by Proposition \ref{Prop:BK_main}, we have a $\C[t]/(t^2)\otimes \C[[h]]$-linear isomorphism
$\tilde{\D}^1_h\xrightarrow{\sim}\tilde{\D}^2_h$ that is the identity modulo $h$.
So we have an automorphism of $\tilde{\D}^1_h$ that  modulo $h$ coincides
with $1+t\delta_0$. We can write $\alpha$ as $\alpha_0+t\alpha_1$, where $\alpha_0,\alpha_1$
are maps $\D_h\rightarrow \D_h$.  Then $\alpha_0^{-1}\circ \alpha_1$
is a derivation of $\D_h$ lifting $\delta_0$.
\end{proof}

\subsection{Symplectic singularities, their $\Q$-terminalizations and quantizations}\label{SS_symp_sing}
The definition of a conical symplectic singularity as well as basic examples were
recalled in Section \ref{SS_intro_HC}. In this section we will study some further
properties of conical symplectic singularities and their quantizations.

Let $Y$ be a conical symplectic singularity. Let us recall
the notation: $\Gamma, \mathcal{L}_i,\Gamma_i, \tilde{\h}_i^*, i=1,\ldots,k, \h_j^*, j=0,\ldots,k$
from Section \ref{SS_HC_symp_sing}.

First, let us discuss covers. Let $\hat{Y}^0$ be a finite \'{e}tale cover of
$Y^{reg}$. Then $\C[\hat{Y}^0]$ is a finitely generated algebra. This follows from
the Stein factorization for $\hat{Y}^0\rightarrow Y$. We set
$\hat{Y}:=\operatorname{Spec}(\C[\hat{Y}^0])$. This is an affine Poisson variety.

The proof of the following lemma was explained to me by Dmytro Matvieievskyi.

\begin{Lem}\label{Lem:symp_sing_cover}
The Poisson variety $\hat{Y}$ is a conical symplectic singularity.
\end{Lem}
\begin{proof}
By the construction,  $
\operatorname{codim}_{\hat{Y}}(\hat{Y}\setminus \hat{Y}^0)\geqslant 2$.
So $\hat{Y}^{reg}$ is symplectic.

Let us show that $\hat{Y}$ has symplectic singularities. By a result of Namikawa,
\cite[Theorem 6]{Namikawa_extension}, it is enough to show that $\hat{Y}$ has rational
Gorenstein singularities. The latter follows from \cite[Theorem 6.2]{Broer}.

Now to prove that $\hat{Y}$ is conical we just need to observe that the action of
$\C^\times$ on $Y^{reg}$ lifts to $\hat{Y}^0$ perhaps after replacing $\C^\times$
with a cover.
\end{proof}

Let us discuss certain partial Poisson resolutions of $Y$: $\Q$-factorial terminalizations
($\Q$-terminalizations for short).  These are  Poisson partial resolutions $\rho:X\rightarrow Y$, where
$X$ is normal and has the following two properties:
\begin{enumerate}
\item The variety $X$ is $\Q$-factorial: every Weil divisor of $X$ is $\Q$-Cartier, meaning that some its positive integral multiple is Cartier.
\item $\operatorname{codim}_X X^{sing}\geqslant 4$. Namikawa proved that, in the present situation,
this is equivalent to $X$ being terminal.
\end{enumerate}
The action of $\C^\times$ on $Y$ then lifts to $X$ by a result of Namikawa.
See \cite[Proposition 2.1]{SRA_der} for details.

Note that $\rho$ is an isomorphism over $Y^{reg}$ and is a resolution of singularities over $Y^{sreg}=Y^{reg}
\cup\bigsqcup_{i=1}^k \mathcal{L}_i$. Let us record the following fact for the future use,
see the proof of \cite[Proposition 1.11]{Namikawa_Poisson}.

\begin{Lem}\label{Lem:cohom_vanish}
We have $\C[X^{reg}]=\C[Y]$ and $H^i(X^{reg},\mathcal{O})=0$ for $i=1,2$.
\end{Lem}

Now let us discuss filtered quantizations following \cite{BPW,orbit},
these results are reviewed, for example, in \cite[Section 3.2]{orbit}.
The space $\h^*_Y$ mentioned in  Section \ref{SS_HC_symp_sing} is
identified with $H^2(X^{reg},\C)$, \cite{Namikawa_Poisson}.
So, by the results recalled in Section \ref{SS_quant_classif_sympl},
to $\lambda\in \h_Y^*$ we can assign the graded formal quantization
$\D^\circ_{\lambda h}$ of $X^{reg}$. Set $\A_{\lambda h}:=\Gamma(\D^\circ_{\lambda h})$.
Lemma \ref{Lem:cohom_vanish} then implies that $\A_{\lambda h}$ is a graded formal
quantization of $Y$. Let $\A_{\lambda h, fin}$ denote the subalgebra
of $\C^\times$-finite elements in $\A_{\lambda h}$. We set $\A_{\lambda}:=
\A_{\lambda h,fin}/(h-1)$.

Let $\iota:X^{reg}\hookrightarrow X$ denote the natural inclusion. Let us write
$\D_{\lambda h}$ for $\iota_* \D^\circ_{\lambda h}$. This is a graded formal
quantization of $X$. Moreover, $X$ has a universal graded Poisson deformation
$X_{\h}$ over $\h^*_Y$ and $\D_{\lambda h}=\D_{\h,h}\otimes_{\C[\h^*_Y][[h]]}\C[[h]]$
where $\D_{\h,h}$ is the canonical quantization of $X_{\h}/\h_Y^*$
and the homomorphism $\C[\h_Y^*][[h]]\rightarrow \C[[h]]$ is given by $h\mapsto h, \alpha\mapsto
\langle\alpha,\lambda\rangle h$ for $\alpha\in \h_Y$.

Some quantizations $\A_{\lambda},\A_{\lambda'}$ for  different $\lambda,\lambda'$ are isomorphic
(while $\D_{\lambda h},\D_{\lambda' h}$ are not). To explain when this happens
we need the Namikawa-Weyl group $W_Y$ defined in \cite{Namikawa2}.
Recall the simply laced Weyl group $\tilde{W}_i$ associated with $\Gamma_i$.
The group $\pi_1(\mathcal{L}_i)$ acts on $\tilde{W}_i$ by diagram automorphisms.
We set $W_i:=\tilde{W}_i^{\pi_1(\mathcal{L}_i)}$, this is a crystallographic
reflection group acting faithfully on $\h_i^*$. Then $W_Y:=\prod_{i=1}^k W_i$.
It is not difficult to show that $\A_{\lambda}\cong \A_{\lambda'}$  if
$\lambda'\in W_Y\lambda$, this follows from \cite[Theorem 3.4]{orbit}.

\begin{Ex}\label{Ex:nilp_cone_quant}
Let $\g$ be a semisimple Lie algebra and $Y=\mathcal{N}$ be the nilpotent cone in $\g$.
Its quantizations are the central reductions of $U(\g)$. Namely, recall that under
the Harish-Chandra isomorphism the center $Z$ of $U(\g)$ gets identified with $\C[\h^*]^W$,
where $\h,W$ are the Cartan space and the Weyl group of $\g$. For $\lambda\in \h^*$
define the central reduction $\U_\lambda$ of $U(\g)$ by $\U_\lambda=U(\g)/U(\g)\mathfrak{m}_\lambda$,
where we write $\mathfrak{m}_\lambda$ for the maximal ideal of $Z$ corresponding to $\lambda$.
We note that $\h_Y=\h, W_Y=W$. Indeed, this reduces to the case when $\g$ is simple. In that case,
we have a unique codimension $2$ symplectic leaf a.k.a. the subregular orbit. The slice to
that orbit in $Y$ has the same type as $\g$ when $\g$ is simply laced and the same type as the
unfolding of the diagram of $\g$ else (for example for type $B_n$ for $n>1$ we get $A_{2n-1}$).
In the non-simply laced  case, $\pi_1$ acts via the group of diagram
automorphisms that folds that diagram.

Note that $\U_\lambda$ is the filtered quantization
of $\C[Y]$ corresponding to $\lambda\in \h_Y^*$.
\end{Ex}

\begin{Ex}\label{Ex:Kleinian_quantizations}
We proceed with $Y=\C^2/\Gamma$, where $\Gamma$ is a finite subgroup of $\operatorname{SL}_2(\C)$.
Pick $c\in (\C\Gamma)_1^\Gamma$ (recall that this means that $c=1+\sum_{\gamma\neq 1}c_\gamma \gamma$), where $\gamma\mapsto
c_\gamma: \C\Gamma\setminus \{1\}\rightarrow \C$ is a $\Gamma$-invariant function.
Consider the Crawley-Boevey-Holland  algebra $H_c:=\C\langle x,y\rangle\#\Gamma/(xy-yx-c)$.
Let $e\in \C\Gamma$ be the averaging idempotent. Then we can consider the spherical subalgebra
$eH_ce$ (with unit $e$). It was explained in   Section \ref{SS_Klein_HC} how to get
$\lambda_c\in \h_Y^*=\tilde{\h}_i^*$ from $c$.

It turns out that we have $\A_{\lambda_c}\cong eH_c e$ for all $c$.
In order to prove this we first note that $\A_\lambda$ is obtained
from $\U_\lambda$ (for simply laced $\g$ of the same type as $\Gamma$)
via the quantum slice construction (see e.g. \cite[Section 3.2]{perv}) applied
to the subregular orbit. The isomorphism $\A_{\lambda_c}\cong eH_c e$
 follows, for example, from \cite[Theorem 6.2.2]{quant}
combined with \cite[Theorem 5.3.1]{quant}.
The parameter $\lambda$ is recovered from the quantization uniquely up to the $W_Y$-conjugacy,
where $W_Y:=\tilde{W}_i$.
\end{Ex}

The construction of Example \ref{Ex:Kleinian_quantizations} has the following useful and
elementary corollary.

\begin{Cor}\label{Cor:CBH_invar}
Let $\Gamma'\subset \Gamma$ be a normal subgroup and assume $c\in (\C\Gamma)_1^\Gamma\cap\C\Gamma'$.
So we have quantizations $\A'_c$ of $\C^2/\Gamma'$ and $\A_c$ of $\C^2/\Gamma$.
Then $\Gamma/\Gamma'$ acts on $\A_c$ by automorphisms and the quantizations
$\A'_c$ and $(\A_c)^{\Gamma/\Gamma'}$ of $\C^2/\Gamma$ are isomorphic.
\end{Cor}

Let us explain how to recover $\lambda$ from $\A_\lambda$ in the case
of a general conical symplectic singularity, \cite[Remark 3.6]{orbit}.  We can write $\lambda$ as $(\lambda_0,\ldots,\lambda_k)$
with $\lambda_i\in \h_i^*$. We are going to explain the meaning of parameters
$\lambda_0,\ldots,\lambda_k$. The parameter $\lambda_0\in H^2(Y^{reg},\C)$ is the period of the microlocalization
$\A_{\lambda h}|_{Y^{reg}}$.

The parameters $\lambda_i$ (defined up to $W_i$-conjugacy) are recovered from
the restriction of $\A_{\lambda h}$ to the formal neighborhood of $y_i\in \mathcal{L}_i$.
Namely, consider the completion $\A_{\lambda h}^{\wedge_{y_i}}$ with respect to the maximal
ideal that is obtained as the inverse image of the maximal ideal of $y_i$ under
the projection $\A_{\lambda h}\twoheadrightarrow \C[Y]$.

Now assume that $d$ is even (we can always replace $d$ with a multiple by rescaling the
$\C^\times$-action). Consider the symplectic vector space $V:=T_{y_i}\mathcal{L}_i$ and form
the homogeneous Weyl algebra $\Weyl_h:=T(V)[h]/(u\otimes v -v\otimes u-h\omega(u,v))$
with $V$ in degree $d/2$. We can also form $\underline{\A}_{\lambda_i h}$, the homogeneous
version of the quantization of $\C^2/\Gamma_i$ with parameter $\lambda_i$.

The following result is a special case of \cite[Lemma 3.3]{perv}, it explains the meaning of
$\lambda_i$ (up to the $W_i$-conjugacy).

\begin{Lem}\label{Lem:compl_iso}
We have a $\C[[h]]$-linear isomorphism $\A_{\lambda h}^{\wedge_{y_i}}\cong \left(\Weyl_h
\otimes_{\C[h]}\underline{\A}_{\lambda_i h}\right)^{\wedge_0}$.
\end{Lem}

Finally, let us explain the classification results for filtered quantizations of $Y$
\cite[Theorem 3.4]{orbit}.

\begin{Prop}\label{Prop:quant_classif}
Every filtered quantization of $\C[Y]$
is of the form $\A_\lambda$ for some $\lambda\in \h_Y^*$.
\end{Prop}

\subsection{Harish-Chandra and Poisson bimodules}\label{SS_Poisson_HC}
Let $X$ be a Poisson scheme.  By a coherent Poisson $\mathcal{O}_X$-module we mean
a coherent sheaf $\mathcal{M}$ of $\mathcal{O}_X$-modules equipped with a map
of sheaves (of vector spaces) $\{\cdot,\cdot\}: \mathcal{O}_X\otimes_{\C}\mathcal{M}
\rightarrow \mathcal{M}$ satisfying the Leibnitz and Jacobi identities
(that are special cases of (2) and (3) below).

Let $X$ come equipped with an action of $\C^\times$ that is compatible
with the Poisson bracket on $\Str_X$ in the following way: there is
a positive integer $d$ such that $t.\{\cdot,\cdot\}=t^{-d}\{\cdot,\cdot\}$
for all $t\in \C^\times$. We say that a coherent Poisson module $\M$ is {\it graded} if
it is $\C^\times$-equivariant (as a coherent sheaf) and $\C^\times$ rescales the bracket
$\Str_X\otimes \M\rightarrow \Str_X$ by $t\mapsto t^{-d}$.

Now let $Y$ be a conical symplectic singularity and $X=Y^{reg}$.
In this case we can fully classify graded coherent  Poisson modules
on $Y^{reg}$ following \cite{B_ineq}. Recall the finite group
$\Gamma=\pi_1^{alg}(Y^{reg})$. Let $\tilde{Y}^0$ denote
the universal algebraic cover of $Y^{reg}$ (with Galois group $\Gamma$)
and $\pi:\tilde{Y}^0\twoheadrightarrow Y^{reg}$ be the quotient map.
Then we have the following result established in the proof of  \cite[Lemma 3.9]{B_ineq}.

\begin{Lem}\label{Lem:Poisson_classif}
The following statements are true:
\begin{enumerate}
\item Every graded coherent  Poisson $\Str_{\tilde{Y}^0}$-module is the direct sum of
several copies of $\Str_{\tilde{Y}^0}$ (with grading shifts).
\item The functor $\pi^*$ defines an equivalence between the category graded coherent Poisson
$\Str_{Y^{reg}}$-modules and the category of $\Gamma$-equivariant graded coherent
$\Str_{\tilde{Y}^0}$-modules. The quasi-inverse is given by $\pi_*(\bullet)^\Gamma$.
In particular, every coherent graded Poisson $\Str_{Y^{reg}}$-module is
semisimple. Up to a grading shift,
the simple graded coherent Poisson $\Str_{Y^{reg}}$-modules are classified by the irreducible representations
of $\Gamma$: the Poisson module corresponding to an irreducible representation $\tau$
is $\Hom_{\Gamma}(\tau, \pi_* \Str_{\tilde{Y}^0})$.
\end{enumerate}
\end{Lem}

Now let $X$ again be a Poisson scheme and $\D_h$ be its formal quantization.
We will review the definition of a coherent module over a formal quantization
below in Section \ref{SS_coh_pushforward}.
Following \cite[Section 3.3]{sraco} define the notion of a coherent Poisson  $\D_h$-bimodule.
By definition, this is
a sheaf $\M_h$ of $\D_h$-bimodules on $X$, coherent as a sheaf of left $\D_h$-modules,
that is equipped with a bracket
map $\{\cdot,\cdot\}: \D_h\otimes_{\C[[h]]}\M_h\rightarrow \M_h$. This bracket
map is supposed to satisfy the following conditions (where for local sections $a,b$ of $\D_\hbar$
we write $\{a,b\}:=\frac{1}{h}[a,b]$):
\begin{enumerate}
\item For local sections $a$ of $\D_\hbar$ and $m$ of $\M_\hbar$
we have $am-ma=h \{a,m\}$.
\item The Jacobi identity: $\{\{a,b\},m\}=\{a,\{b,m\}\}-\{b,\{a,m\}\}$.
\item The Leibniz identities:
\begin{align*}
& \{ab,m\}=\{a,m\}b+a\{b,m\},\\
& \{a,bm\}=\{a,b\}m+b\{a,m\},\\
& \{a,mb\}=\{a,m\}b+m\{a,b\}.
\end{align*}
\end{enumerate}
Note that when $h$ acts on $\M_h$ by $0$ what we get is precisely the notion
of a coherent Poisson $\Str_X$-module from above. As the other extreme, assume $\M_h$ is $\C[[h]]$-flat. Here (1) allows to recover $\{\cdot,\cdot\}$
from the bimodule structure on $\M_h$. Note that the $h$-adic filtration on $\M_h$
is automatically complete and separated, this is true for all coherent
$\D_h$-modules.

When $\C^\times$ acts on $\D_h$ as above,
we can talk about graded coherent Poisson $\D_h$-bimodules.

We now proceed to  HC bimodules.
Let  $X:=\operatorname{Spec}(A)$.
Let $\D_h$ be a graded formal quantization of $X$ and let $\A_h$ be the $\C^\times$-finite
part of $H^0(X,\D_h)$. Then $\A:=\A_h/(h-1)$ is a filtered quantization of $A$.
Set $\hbar:=\sqrt[d]{h}$. We can form the Rees
algebra $R_\hbar(\A)$. Then $R_\hbar(\A)$ is naturally identified with
$\C[\hbar]\otimes_{\C[h]}\A_h$. Now let $\B$ be a HC $\A$-bimodule. Choose a
good filtration on $\B$. Then the Rees bimodule $R_\hbar(\B)$ is a
$R_\hbar(\A)$-bimodule and also a Poisson $\A_h$-bimodule. We will call
such bimodules graded   Poisson $R_\hbar(\A)$-bimodules.

Let us introduce some notation for HC bimodules in the case when $\A$
is a quantization of a conical symplectic singularity $Y$.
Denote the category of HC $\A$-bimodules by $\HC(\A)$.

Let $\B\in \HC(\A)$.
It is a classical fact that the support of the $\C[Y]$-module $\gr\B$ in $Y$
is independent of the choice of a good filtration.
This support is called the {\it associated variety} of $\B$ and is denoted by $\VA(\B)$.
This is a Poisson subvariety.

The following lemma is also standard.

\begin{Lem}\label{Lem:tens_hom_HC}
Let $\B,\B'\in \HC(\A)$. Then $\B\otimes_\A
\B', \Hom_\A(\B,\B'),\Hom_{\A^{opp}}(\B,\B')\in \HC(\A)$
and the associated varieties of these bimodules are contained in
$\VA(\B)\cap \VA(\B')$.
\end{Lem}

It follows, in particular, that $\VA(\B)\neq Y$ if and only if
$\B$ is not faithful as a left (equivalently, right) bimodule.
Let us write $\overline{\HC}(\A)$ for the Serre quotient
$$\HC(\A)/\{\B\in \HC(\A)| \VA(\B)\neq Y\}.$$
We call $\overline{\HC}(\A)$ the category of HC bimodules {\it with full support}.
By Lemma \ref{Lem:tens_hom_HC}, this is a rigid monoidal category.

We now turn to a connection between the categories of HC bimodules and of
the graded Poisson $R_\hbar(\A)$-bimodules. The functor $\B_\hbar\rightarrow
\B_\hbar/(\hbar-1)\B_\hbar$ maps from the category of graded Poisson bimodules
to the category of HC bimodules. It is not difficult to see that it is
a Serre quotient functor, the kernel consists of the $\hbar$-torsion modules.

The connection described in the previous paragraph can be extended to non-affine
varieties. Namely, take a  graded formal quantization $\D_{h}$ of a
normal variety $X$. We can form the microlocal sheaves $\D_{h,fin}$ of $\C^\times$-finite
sections of $\D_h$ and $\D:=\D_{h,fin}/(
\hbar-1)\D_{h,fin}$ on $X$
(where ``microlocal'' means that the sections are only defined on the $\C^\times$-stable
open subsets). We can define the notion of a HC bimodule over $\D$ as a bimodule
that has a complete and separated good filtration. The category of HC $\D$-bimodules
is the quotient of the category of graded coherent Poisson $\D_\hbar$-bimodules by the full subcategory
of $\hbar$-torsion bimodules.

Finally, we need to recall the construction of restriction functors for HC bimodules
considered in this (and greater) generality in \cite[Section 3.3]{perv}. We use the setting of
Lemma \ref{Lem:compl_iso}. Note that both algebras $R_\hbar(\A_\lambda),
R_\hbar(\Weyl\otimes \underline{\A}_{\lambda_i})$ come equipped with the Euler derivations
coming from the gradings. We will denote these derivations by $\mathsf{eu},
\mathsf{eu}'$. The following claim was obtained in the proof of \cite[Lemma 3.3]{perv}.

\begin{Lem}\label{Lem:compl_iso1}
Under the isomorphism of Lemma \ref{Lem:compl_iso}, the derivations
$\mathsf{eu}$ and $\mathsf{eu}'$ differ by a derivation of the form
$\frac{1}{h}\operatorname{ad}(a)$ for $a\in \A_{\lambda h}^{\wedge_{y_i}}$.
\end{Lem}

Now we recall the  construction of  a functor $\HC(\A_\lambda)\rightarrow \HC(\underline{\A}_{\lambda_i})$
that we will denote by $\bullet_{\dagger,i}$ (see \cite[Section 3.3]{perv}).
Consider the completion $R_\hbar^{\wedge}(\B)$ at $y_i$, this is an $R_\hbar^{\wedge}(\A_{\lambda})$-bimodule.
It comes equipped with the derivation $\mathsf{eu}$ that is compatible with
the eponymous derivation of $R_\hbar^{\wedge}(\A_{\lambda})$.
By Lemma \ref{Lem:compl_iso}, $R_\hbar^\wedge(\B)$ can be viewed as an
$R^{\wedge}_\hbar(\Weyl\otimes \underline{\A}_{\lambda_i})$-bimodule.
We define an operator $\mathsf{eu}'$ on $R^\wedge_\hbar(\B)$ as follows:
$\mathsf{eu}'=\mathsf{eu}+\frac{1}{h}\operatorname{ad}a$, where $a$
is as in Lemma \ref{Lem:compl_iso1}.
Note that $\mathsf{eu}'$ is compatible with the derivation $\mathsf{eu}'$ of the algebra
$R_\hbar(\Weyl\otimes \underline{\A}_{\lambda^i})$.

It was shown in \cite[Section 3.3]{perv} that $R^\wedge_\hbar(\B)$ splits as $R_\hbar^{\wedge}(\Weyl)\widehat{\otimes}_{\C[[\hbar]]}
\underline{\B}_\hbar$, where $\underline{\B}_\hbar$ is the centralizer of $R_\hbar^{\wedge}(\Weyl)$
in    $R^\wedge_\hbar(\B)$. In particular, $\mathsf{eu}'$ preserves $\underline{\B}_\hbar$.
Consider the subspace $\underline{\B}_{\hbar,fin}$ of all $\mathsf{eu}'$-finite elements.
This is a $R_\hbar(\underline{\A}_{\lambda_i})$-sub-bimodule. Set $\B_{\dagger,i}:=\underline{\B}_{\hbar,fin}/(\hbar-1)$.

It was shown in \cite[Sections 3.3,3.4]{HC}, that this construction indeed gives a functor $\HC(\A_{\lambda})
\rightarrow \HC(\underline{\A}_{\lambda_i})$. This functor is exact and tensor.
On the level of associated graded bimodules the functor becomes (the algebraization of)
the restriction of the Poisson bimodule to the slice.
In particular, it descends to $\overline{\HC}(\A_\lambda) \rightarrow \overline{\HC}(\underline{\A}_{\lambda_i})$.

\subsection{Pushforwards of coherent $\mathcal{D}_h$-modules}\label{SS_coh_pushforward}
Let $X$ be a normal Poisson variety and $\D_h$ be its formal quantization. Recall that a
$\D_h$-module $M_h$ is called {\it coherent} if there is an open affine cover $X=\bigcup U_i$
such that $M_h|_{U_i}$ is obtained by microlocalizing a
finitely generated $H^0(U_i,\D_h)$-module. In this case,
for every open affine subvariety $U\subset X$, the restriction $M_h|_U$ is the microlocalization
of $H^0(U, M_h)$. Note that every coherent $\D_h$-module is automatically complete and separated
in the $h$-adic topology.
%Conversely, if a $\D_h$-module $M_h$ is complete and separated in the
%$h$-adic topology and $M_h/hM_h$ is a coherent $\Str_X$-module, then $M_h$ is coherent.

Now let $\iota: X^0\hookrightarrow X$ be an open embedding and let $M^0_h$ be a coherent
$\D_h^0:=\D_h|_{X^0}$-module. We set $M^0_{h,k}:=M^0_h/h^k M^0_h$. When $k=1$, we write
$M^0$ instead of $M^0_{h,1}$. We will assume the following  condition:
\begin{itemize}
\item[(*)] The pushforward $\iota_*(M^0)$ is a coherent $\Str_X$-module.
\end{itemize}

We want to get a sufficient condition for $\iota_* M^0_h$ to be coherent.
We note that for a $\D_h^0/(h^k)$-modules
it makes sense to speak about quasi-coherent modules, and the push-forward maps quasi-coherent
modules to quasi-coherent ones -- for the same reason as for the quasi-coherent
$\mathcal{O}_X$-modules. If  (*) holds, then $\iota_*(M^0_{h,k})$
is coherent for all $k$. However, $\iota_*(M^0_h)$ may fail to be coherent: the problem is that,
for some open affine subset $U\subset X$, the $H^0(U,\D_h)$-module $H^0(U\cap X^0, M^0_h)$ may not be large enough (in fact, it can be zero even if $M^0_h$ is nonzero).

Let $U\subset X$ be an open affine subvariety. Set $U^0:=U\cap X^0$. We write $H^0(U^0, M^0)_k$
for the image of $H^0(U^0, M^0_{h,k})$ in $H^0(U^0,M^0)$. Note that the submodules
$H^0(U,M^0)_k$ form a decreasing sequence.

%Consider the following condition:
%
%\begin{itemize}
%\item[(c($U$))] The sequence $H^0(U^0,M^0)_k$ stabilizes.
%\end{itemize}

%Note that the localization of the stable module $H^0(U,M^0)_k$ to $U\cap X^0$ coincides with

The following lemma gives a sufficient condition for $\iota_* M^0_h$ to be coherent.

\begin{Lem}\label{Lem:coherence_criterium}
Suppose (*) holds. Let $X=\bigcup_{i=1}^\ell U_i$ be an open affine cover. Suppose that,
for every $i$, the sequence $H^0(U_i^0, M^0)_k$ stabilizes. Then $\iota_* M^0_h$
is coherent.
\end{Lem}
\begin{proof}
Note that the condition that the sequence $H^0(U_i^0,M^0)_k$ stabilizes is equivalent
to the Mittag-Leffler (ML) condition for the inverse system of $H^0(U_i,\D_h)$-modules
$H^0(U_i^0, M^0_{h,k})$. Since $M^0_{h}=\varprojlim M^0_{h,k}$ we see that
$H^0(U^0_i, M^0_h)=\varprojlim H^0(U^0_i, M^0_{h,k})$. Note that, thanks to (*),
$H^0(U^0_i, M^0_h)$ is a finitely generated $H^0(U_i, \D_h)$-module. It remains to
show that $(\iota_* M^0_h)|_{U_i}$ is obtained by microlocalizing
$H^0(U^0_i, M^0_h)$. This is equivalent to the condition that,
for every $f\in \C[U_i]$, the module $H^0(U^0_{i,f},M^0_h)$
is the microlocalization of $H^0(U^0_{i},M^0_h)$ at $f$
to be denoted by $H^0(U^0_{i},M^0_h)[f^{-1}]$.
This condition holds if we replace $M^0_h$ with $M^0_{h,k}$.
Then, by ML, it holds for $M^0_h$.
\end{proof}

\subsection{Etale lifts of Poisson bimodules}\label{SS_lift_etale}
Let $\D_h$ be a formal quantization of a smooth symplectic variety $X$.
The goal of this section is to make sense of pullbacks of Poisson $\D_h$-modules
under \'{e}tale morphisms of symplectic varieties.
We will do so by considering jet bundles for Poisson bimodules. Recall that the jet bundles of $\Str_X$ and of its formal quantization
were discussed in Section \ref{SS_Per_NC}.

We can define the notion of a coherent Poisson $\jet\D_{h}$-bimodule $\mathfrak{B}_h$.
By definition, this is a  pro-coherent sheaf of $\Str_X$-modules that comes with
\begin{enumerate}
\item a  $\jet\D_{h}$-bimodule
structure making it into a locally finitely generated $\jet\D_{h}$-module,
\item a $\C[[h]]$-linear flat connection $\nabla$ that is compatible
with the flat connection on $\jet\D_{h}$,
\item  and a flat bracket map
$\{\cdot,\cdot\}: \jet\D_{h}\otimes_{\C}\mathfrak{B}_h\rightarrow \mathfrak{B}_h$
satisfying the axioms (1)-(3) listed in Section \ref{SS_Poisson_HC},
\item such that $\mathfrak{B}_h$ is complete and separated in the $h$-adic topology and,
moreover, for each $k\in \Z_{>0}$, the quotient $h^{k-1}\mathfrak{B}_{h}/h^k \mathfrak{B}_h$
is the jet bundle of a coherent Poisson $\Str_X$-module.
\end{enumerate}
Now let $\B_h$ be a Poisson $\D_{h}$-bimodule. Similarly to what was done in Section
\ref{SS_Per_NC}, we can form the jet bundle $\jet\B_h$, which is  a pro-coherent sheaf
on $X$ with a flat connection. It is easy to see that this is a coherent Poisson $\jet\D_{h}$-bimodule.
%Our next task is to characterize the Poisson $\jet\D_{h}$-bimodules that arise in this way.

\begin{Lem}\label{Lem:Poisson_equiv}
The functors $\B_h\mapsto \jet\B_h$ and $\mathfrak{B}_h\mapsto \mathfrak{B}_h^\nabla$
are equivalences between the category of coherent Poisson bimodules over $\D_{h}$ and  the category of coherent  Poisson $\jet\D_{h}$-bimodules.
\end{Lem}
\begin{proof}
We write $\bullet^\nabla$ for the functor of taking flat sections.
By the construction, $(\jet \mathcal{B}_h)^\nabla=\B_h$.
On the other hand let $\mathfrak{B}_h$ be a coherent Poisson $\jet\D_h$-module.
Note that $\mathfrak{B}_h^{\nabla}$
is complete and separated in the $h$-adic topology and $\mathfrak{B}_h^\nabla/h \mathfrak{B}_h^\nabla
\hookrightarrow (\mathfrak{B}_h/h \mathfrak{B}_h)^\nabla$. So $\mathfrak{B}_h^{\nabla}$ is a coherent
Poisson $\D_h$-bimodule.

Now we need to show that the jet bundle $\jet(\mathfrak{B}_h^\nabla)$ is functorially isomorphic to $\mathfrak{B}_h$. First of all, observe that the left $\jet\D_{h}$-action on $\jet(\mathcal{B}_h)$
for any $\mathcal{B}_h$ induces an isomorphism $\jet\D_{h}\widehat{\otimes}_{\D_h}\mathcal{B}_h
\xrightarrow{\sim} \jet\mathcal{B}_h$. This gives rise to a Poisson bimodule homomorphism
$\jet(\mathfrak{B}_h^\nabla)
\rightarrow \mathfrak{B}_h$. We want to show that it is an isomorphism.

We start by showing that, for any vector bundle $\mathcal{V}$ on $X$, we have
\begin{equation}\label{eq:jet_vanishing} R^1(\jet \mathcal{V})^\nabla=0. \end{equation}
This will follow from a  stronger statement:
$R\mathcal{H}om_{D_X}(\Str_X,\jet \mathcal{V})=\mathcal{V}$
for all vector bundles $\mathcal{V}$.  The latter equality is standard.

Thanks to (\ref{eq:jet_vanishing}), for all $k$, we have
\begin{align*}
&\mathfrak{B}_h^\nabla/
h^k \mathfrak{B}_h^\nabla\xrightarrow{\sim}
(\mathfrak{B}_h/h^k \mathfrak{B}_h)^\nabla,\\
&\jet\left((\mathfrak{B}_h/h^k \mathfrak{B}_h)^\nabla\right)\xrightarrow{\sim} \mathfrak{B}_h/h^k \mathfrak{B}_h.
\end{align*}
Since $\mathfrak{B}_h$ is complete and separated in the $h$-adic topology, we deduce
that $\jet(\mathfrak{B}_h^\nabla)
\xrightarrow{\sim} \mathfrak{B}_h$.

%It remains to show that it induces  isomorphisms fiberwise, i.e., that
%the natural homomorphism
%\begin{equation}\label{eq:nat_iso}(\mathfrak{B}_\hbar^\nabla)^{\wedge_x}\rightarrow %\mathfrak{B}_{\hbar,x}\end{equation}
%is an isomorphism. Note that $\mathfrak{B}_\hbar^\nabla=\varprojlim %(\mathfrak{B}_\hbar/\hbar^k\mathfrak{B}_\hbar)^\nabla$ and hence
%$$(\mathfrak{B}_\hbar^\nabla)^{\wedge_x}=\varprojlim %\left((\mathfrak{B}_\hbar/\hbar^k\mathfrak{B}_\hbar)^\nabla\right)^{\wedge_x}.$$
%So it is enough to show that (\ref{eq:nat_iso}) is an isomorphism when $\mathfrak{B}_\hbar$
%is annihilated by $\hbar^k$ for some $k$, which we will assume until the end of the proof.

%Note that, by (*), $\mathfrak{B}_\hbar$ is filtered with $\jet \mathcal{V}$ for various
%vector bundles $\mathcal{V}$. For $\mathfrak{B}_\hbar:=\jet \mathcal{V}$,
%(\ref{eq:nat_iso}) is clear. So it remains to show that $\bullet^\nabla$ is exact on the category of
%Poisson $\jet\D_{\hbar}$-bimodules that are annihilated by $\hbar^k$ and are filtered by the bundles of %the form $\jet\mathcal{V}$.
%This amounts to showing
\end{proof}

Lemma \ref{Lem:Poisson_equiv} allows to define the pullback of Poisson bimodules under  an \'{e}tale morphism intertwining symplectic forms.
Namely, let $\varphi:X^1\rightarrow X^2$ be such a morphism.  Let $\D^1_{h}$ be a formal
quantization of $X^1$. Then $\varphi^* \jet\D^1_{h}$ is a quantum jet bundle on $X^2$.
Passing to the sheaf of flat sections, we get a formal quantization
$\D^2_{h}$ of $X^2$ with $\jet \D^2_{h}=\varphi^* \jet\D^1_{h}$.

Now let $\B^1_{h}$ be a Poisson $\D^1_{h}$-bimodule. Then $\varphi^* \left(\jet \B^1_h\right)$
is a coherent Poisson $\jet\D_h$-module. We set $\varphi^* \B^1_h:=
(\varphi^* \jet \B^1_h)^\nabla$. The following properties are straightforward from the construction.

\begin{Lem}\label{Lem:Poisson_bim_etale_pullback}
The following claims hold:
\begin{itemize}
\item If $\B^1_h$ is annihilated by $h$ (so that $\B^1_h$ is a vector bundle
with a flat connection), then $\varphi^* \B^1_h$
is the usual pull-back of a vector bundle with a flat connection.
\item The functor $\varphi^*$ is exact, faithful and $\C[[h]]$-linear.
\item We have a natural isomorphism of left $\D^2_{h}$-modules
$$\varphi^* \mathcal{B}^1_h\cong
\D^2_{h}\widehat{\otimes}_{\varphi^{-1}\D^1_{h}}\varphi^{-1}\mathcal{B}^1_{h}.$$
\end{itemize}
\end{Lem}

\section{Extending  Poisson bimodules from $Y^{reg}$ to $Y$}\label{S_extension}
\subsection{Main result}
Let $Y$ be a conical symplectic singularity and $\A$ its filtered quantization.

Let $\D_\hbar$ be the microlocalization of $R_\hbar(\A)$ to $Y$ so that
$\D_\hbar=\C[\hbar]\otimes_{\C[h]}\D_h$, where $\D_h$ is the microlocalization of
$\A_{\lambda h}$ and $\hbar^d=h$. Set $\A_\hbar:=\Gamma(\D_\hbar)$, this is the $\hbar$-adic
completion of $R_\hbar(\A)$.
We write $\D_{\hbar}^{reg}$ for the restriction of $\D_{\hbar}$
to $Y^{reg}$. Recall that we also consider the open subvariety $Y^{sreg}\subset Y$,
defined by $Y^{sreg}:=Y^{reg}\sqcup \bigsqcup_{i=1}^k\mathcal{L}_i$, where
$\mathcal{L}_1,\ldots,\mathcal{L}_k$ are all codimension $2$ symplectic leaves.
We consider the restriction $\D_{\hbar}^{sreg}$. Also set $Y_i:=\operatorname{Spec}(\C[Y]^{\wedge_y})$,
where $y\in \mathcal{L}_i$ and $\C[Y]^{\wedge_y}$ denote the completion of
$\C[Y]$ at $y$. We write $Y_i^\times$ for
$Y_i\setminus \mathcal{L}_i$. We can restrict $\D_{\hbar}^{sreg}$ to $Y_i$
getting a  formal  quantization $\D_{\hbar}|_{Y_i}$.
We can further restrict this formal quantization to $Y_i^\times$.

Our primary goal in this section is to understand conditions for a
coherent Poisson $\D_{\hbar}^{reg}$-bimodule $\B_\hbar$ that is flat over
$\C[[\hbar]]$ (the only case we are interested in) to extend to
a graded Poisson $\A_{\hbar}$-bimodule. Note that when $\hbar$
acts on $\B_\hbar$ by zero, then we can take $H^0(Y^{reg},\B_\hbar)$ for this extension,
indeed, $\B_\hbar$ is a vector bundle and since $\operatorname{codim}_{Y}Y^{sing}\geqslant 2$,
the global sections of every vector bundle on $Y^{reg}$ is a finitely generated
$\C[Y]$-module.

Here is the main result to be proved in this section.

\begin{Prop}\label{Prop:HC_extension}
%Let $\B_\hbar$ be $\C[[\hbar]]$-flat.
Let $\B_\hbar$ be $\C[[\hbar]]$-flat.  Then the following two conditions are
equivalent:
\begin{enumerate}
\item The restriction of $\Gamma(\B_\hbar)$ to $Y^{reg}$ coincides with $\B_\hbar$.
\item The restriction of $\Gamma(\B_\hbar|_{Y_i^\times})$ to $Y_i^\times$
coincides with $\B_\hbar|_{Y_i^\times}$ for all $i=1,\ldots,k$.
\end{enumerate}
\end{Prop}

We will explain the meaning of $\B_\hbar|_{Y_i^\times}$ below in Section \ref{SS_restriction_formal}.
Note that (1)$\Rightarrow$(2) is relatively easy, while (2)$\Rightarrow$(1) is harder.
%We note that the implication (1)$\Rightarrow$(2) is easy because of the natural
%isomorphism $\Gamma(\B_\hbar)|_{Y_i}\rightarrow \Gamma(\B_\hbar|_{Y_i})$.

The proof is in two steps. First, let $\iota$ denote the inclusion $Y^{reg}\hookrightarrow
Y^{sreg}$. So we get a Poisson $\D^{sreg}_{\hbar}$-bimodule
$\iota_*\B_\hbar$ that is flat over $\C[[\hbar]]$. In Section \ref{SS_ext_codim2},
we will show that condition (2) of the proposition is equivalent to
the claim that $\iota_* \B_\hbar$ is coherent. Then in Section \ref{SS_ext_codim4}
we show that, if (2) holds, then  $\Gamma(\B_\hbar)|_{Y^{reg}}\cong \iota^*\iota_* \B_\hbar$.
This will imply Proposition \ref{Prop:HC_extension}.

\subsection{Construction of $\B_\hbar|_{Y_i^\times}$}\label{SS_restriction_formal}
The goal of this section is to make sense of the Poisson bimodule $\B_\hbar|_{Y_i^\times}$
and study properties of this Poisson bimodule.
Let $\varphi$ denote the morphism $Y_i^\times\rightarrow Y^{reg}$ induced by
the inclusion $Y_i\rightarrow Y$. We define $\B_\hbar|_{Y_i^\times}$ as
$$\D_\hbar|_{Y_i^\times}\otimes_{\varphi^{-1}\D_\hbar} \varphi^{-1}\B_\hbar.$$
This is a coherent $\D_\hbar|_{Y_i^\times}$-module. What we need to do is to construct
the bracket map satisfying conditions (2) and (3) from Section \ref{SS_Poisson_HC}.

Let $U$ be an open affine neighborhood in $X$ of the closed point $y$ in $Y_i$. Note that
we have the bracket map $\D_\hbar(U)\otimes_{\C[[\hbar]]} \B_{\hbar}|_{Y_i^\times}
\rightarrow \B_\hbar|_{Y_i^\times}$ with required properties.

\begin{Lem}\label{Lem:bracket_extension}
This bracket extends to a bracket $\D_\hbar|_{Y_i^\times}\otimes_{\C[[\hbar]]} \B_{\hbar}|_{Y_i^\times}
\rightarrow \B_\hbar|_{Y_i^\times}$ satisfying (2) and (3).
\end{Lem}
\begin{proof}
Consider the case when
$\hbar^k$ annihilates $\B_\hbar$ for some $k>0$. Then we can consider the push-forward $\bar{\B}_\hbar$
of $\B_\hbar$ to $Y$, this is a coherent Poisson $\D_\hbar$-bimodule. Consider the restriction
$\bar{\B}_\hbar^{\wedge_y}$ to $Y_i$. It comes with the bracket with $\D_\hbar^{\wedge_y}$.
Then we can localize the bracket to $Y_i^\times$. This settles the case when $\B_\hbar$
is annihilated by $\hbar^k$. To handle the general case we notice that the bracket is
continuous in the $\hbar$-adic topology and $\B_\hbar|_{Y_i^\times}=\varprojlim (\B_\hbar/\hbar^k \B_\hbar)|_{Y_i^\times}$.
\end{proof}

The following properties are straightforward from the construction.

\begin{Lem}\label{Lem:Poisson_bim_pullback}
The following claims hold:
\begin{itemize}
\item If $\B_\hbar$ is annihilated by $\hbar$ (so that $\B_\hbar$ is a vector bundle
with a flat connection), then $\B^1_\hbar|_{Y_i^\times}$
is the usual pull-back of a vector bundle with a flat connection.
\item The functor $\bullet|_{Y_i^\times}$ is exact, faithful and $\C[[\hbar]]$-linear. In particular,
$\B_\hbar|_{Y_i^\times}=\varprojlim (\B_\hbar/\hbar^k \B_\hbar)|_{Y_i^\times}$.
\end{itemize}
\end{Lem}

\subsection{Extension to codimension $2$}\label{SS_ext_codim2}
%The construction of the previous section allows to construct the restriction
%$\B_\hbar|_{Y_i^\times}$ because the natural inclusion $Y_i^\times
%\rightarrow Y^{reg}$  is \'{e}tale.

The goal of this section is to prove the following lemma.

\begin{Lem}\label{Lem:HC_extension}
Let $\iota:Y^{reg}\hookrightarrow Y^{sreg}$ denote the inclusion.
%Let $\B_\hbar$ be $\C[[\hbar]]$-flat.
 Then the following two conditions are
equivalent:
\begin{enumerate}
\item $\iota_*\B_\hbar$ is coherent.
\item The restriction of $H^0(Y_i^\times,\B_\hbar|_{Y_i^\times})$ to $Y_i^\times$
coincides with $\B_\hbar|_{Y_i^\times}$ for all $i=1,\ldots,k$.
\end{enumerate}
\end{Lem}
\begin{proof}
Let us prove (1)$\Rightarrow$(2). Note that $H^0(Y_i^\times,\B_\hbar|_{Y_i^\times})|_{Y_i^\times}
\hookrightarrow \B_\hbar|_{Y_i^\times}$. So we need to prove this map is surjective. Consider the restriction $(\iota_*\B_\hbar)^{\wedge_y}$ of $\iota_*\B_\hbar$ to $Y_i$. Then
$(\iota_*\B_\hbar)^{\wedge_y}\hookrightarrow H^0(Y_i^\times,\B_\hbar|_{Y_i^\times})$.
On the other hand, $\iota^*\iota_*\B_\hbar\cong \B_\hbar$ and hence  $(\iota_*\B_\hbar)^{\wedge_y}|_{Y_i^\times}\xrightarrow{\sim} \B_\hbar|_{Y_i^\times}$.
This implies (2).

Now we prove (2)$\Rightarrow$(1). We can find an open affine cover $Y^{sreg}=\bigcup_{j=1}^\ell U_i$
such that each $U_j$ intersects at most one codimension $2$ symplectic leaf. We will check the
condition of Lemma \ref{Lem:coherence_criterium} for $M_h^0=\B_\hbar$ and  this cover, this will imply (1). Let $U=U_j$ for some $j$.
So we need to prove that the sequence $H^0(U^{reg}, M^0)_m$ stabilizes. If $U\subset Y^{reg}$,
there is nothing to prove. So let $\mathcal{L}:=U^{sing}$, it is an open subvariety in some $\mathcal{L}_i$. Note that each
$H^0(U^{reg},M^0)_m$ is a  Poisson $\C[U]$-submodule in the finitely generated Poisson
$\C[U]$-module $H^0(U^{reg},M^0)$. Since $H^0(U^{reg},M^0)/H^0(U^{reg},M^0)_m$ is supported on $\mathcal{L}$, we see that  $H^0(U^{reg},M^0)/H^0(U^{reg},M^0)_m$ admits a finite filtration by vector
bundles on $\mathcal{L}$. In particular, it is enough to show that, for $y\in \mathcal{L}$,
the sequence $[H^0(U^{reg},M^0)_m]^{\wedge_y}$ stabilizes. This is done in three steps.
We will use the notation of Lemma \ref{Lem:coherence_criterium}.

{\it Step 1}. We claim that $H^n(U^{reg},M^0)^{\wedge_y}\xrightarrow{\sim}
H^n(Y_i^\times, M^0|_{Y_i^\times})$ for all $n$.
Indeed, we cover $U^{reg}$ with principal open subsets $V_s=U_{f_s}$ for a finite collection $f_s\in \C[U]$.
Let $M:=H^0(U^{reg},M^0)$.  Then   $H^n(U^{reg},M^0)$ is the cohomology of
the Cech complex for $M$ and the cover $V_s$ of $U^{reg}$.  Since the completion functor
is exact, we see that it sends to Cech complex for $M$ to that of $M^{\wedge_y}$.
But $H^n(Y_i^\times, M^0|_{Y_i^\times})$ is the cohomology space for the
Cech complex of $M^{\wedge_y}$. Our isomorphism is proved.

{\it Step 2}. We claim that  $H^n(U^{reg},M^0_{\hbar,p})^{\wedge_y}\xrightarrow{\sim}
H^n(Y_i^\times, M^0_{\hbar,p}|_{Y_i^\times})$ for all $n$ and $p$. This follows by induction on $p$
using  Step 1 and the 5-lemma.

{\it Step 3}. Thanks to Step 2,   we will be done if we know that
the sequence $H^0(Y_i^\times, M^0|_{Y_i^\times})_m$ stabilizes. Note that
$H^0(Y_i^\times, \B_\hbar|_{Y_i^\times})/\hbar H^0(Y_i^\times, \B_\hbar|_{Y_i^\times})
\hookrightarrow H^0(Y_i^\times, M^0|_{Y_i^\times})$. The image is contained in
all  $H^0(Y_i^\times, M^0|_{Y_i^\times})_m$. Thanks to (2), the cokernel is
supported on $\mathcal{L}^{\wedge_y}$ and hence has finite length. So the sequence $H^0(Y_i^\times, M^0|_{Y_i^\times})_m$  indeed stabilizes.
This completes the proof.
\end{proof}

\subsection{Extension to $Y$}\label{SS_ext_codim4}
Let $\iota'$ denote the embedding of $Y^{sreg}$ to $Y$.
The goal of this section is
to prove the following result. Together with Lemma \ref{Lem:HC_extension}, this will
finish the proof of Proposition \ref{Prop:HC_extension}.

\begin{Lem}\label{Lem:HC_extension_codim4}
Let $\B'_\hbar$ be a coherent Poisson $\D_{\hbar}|_{Y^{sreg}}$-bimodule flat over $\C[[\hbar]]$.
Then the restriction of $\Gamma(\B'_\hbar)$ to $Y^{sreg}$ coincides with $\B'_\hbar$.
\end{Lem}
\begin{proof}
It is enough to show that
\begin{itemize}
\item[(*)]
$H^1(Y^{sreg}, \B'_\hbar)$ is a finitely generated
$\A_{\hbar}$-module supported on $Y\setminus Y^{sreg}$.
\end{itemize}

Indeed, if we know (*), then the restriction of $H^0(Y^{sreg}, \B'_\hbar)/\hbar
H^0(Y^{sreg}, \B'_\hbar)$ to $Y^{sreg}$ coincides with $\B'_\hbar/\hbar \B'_\hbar$.
Since $\B'_\hbar$ is flat over $\C[[\hbar]]$, it follows that $H^0(Y^{sreg},\B'_\hbar)|_{Y^{sreg}}
\xrightarrow{\sim} \B'_\hbar$.

Following the proof of \cite[Lemma 5.6.3]{GL},
we see that  (*) will follow once we know that $H^1(Y^{sreg}, \B'_\hbar/\hbar \B'_\hbar)$
is a finitely generated $\C[Y]$-module, automatically  supported on $Y\setminus Y^{sreg}$.
So we proceed to proving that
\begin{itemize}
\item[(**)] $H^1(Y^{sreg}, \B'_\hbar/\hbar \B'_\hbar)$ is
finitely generated.
\end{itemize}

Recall that  $\iota$ denotes the embedding of $Y^{reg}$ into $Y^{sreg}$. Let $\B_0:=\iota_*\iota^*(\B'_\hbar/\hbar \B'_\hbar)$
so that we have an exact sequence
\begin{equation}\label{eq:ses2}0\rightarrow \B'_\hbar/\hbar \B'_\hbar\rightarrow \B_0\rightarrow \mathcal{V}\rightarrow 0,\end{equation}
where $\mathcal{V}$ is a Poisson $\Str_{Y^{sreg}}$-module supported on $Y^{sreg}\setminus Y^{reg}$.

Now consider a part of the long exact sequence induced by (\ref{eq:ses2}):
$$H^0(Y^{sreg},\B'_\hbar/\hbar \B'_\hbar)\rightarrow
H^0(Y^{sreg},\B_0)\rightarrow H^0(Y^{sreg}, \mathcal{V})\rightarrow H^1(Y^{sreg}, \B'_\hbar/\hbar\B'_\hbar)
\rightarrow H^1(Y^{sreg}, \B_0).$$

Note that $\mathcal{V}$ is filtered by vector bundles on $Y^{sreg}\setminus Y^{reg}$.
We have $\operatorname{codim}_{Y^{sing}}Y^{sing}\setminus Y^{sreg}\geqslant 2$.
It follows that for every vector bundle on $Y^{sreg}\setminus Y^{reg}$
its global section is a finitely generated $\C[Y]$-module.
From here one deduces that $H^0(Y^{sreg},\mathcal{V})$
is a finitely generated $\C[Y]$-module supported on $Y^{sing}$.
It follows that the cokernel of $H^0(Y^{sreg},\B_0)\rightarrow H^0(Y^{sreg},\mathcal{V})$  is finitely generated over $\C[Y]$.
So (**) reduces to the claim that $H^1(Y^{sreg},\B_0)$ is finitely generated over $\C[Y]$.

A key step here is to show that $\B_0$ is maximal Cohen-Macaulay on $Y^{sreg}$. Let $\pi$ denote the
quotient morphism $\tilde{Y}\rightarrow \tilde{Y}/\Gamma=Y$, where $\Gamma=\pi_1^{alg}(Y^{reg})$,
 $\tilde{Y}^0$ is the universal cover of
$Y^{reg}$ and $\tilde{Y}=\operatorname{Spec}(\C[\tilde{Y}^0])$. By Lemma \ref{Lem:Poisson_classif},
$\iota^*(\B'_\hbar/\hbar \B'_\hbar)$ is the direct sum of  $\Gamma$-isotypic components in $\pi_*\mathcal{O}_{\tilde{Y}^0}$.
It follows that $\iota'_*\B_0$ is the direct sum of  $\Gamma$-isotypic components of
$\pi_*\mathcal{O}_{\tilde{Y}}$. Lemma \ref{Lem:symp_sing_cover} implies that
$\tilde{Y}$ is Cohen-Macaulay. Therefore $\pi_*\mathcal{O}_{\tilde{Y}}$
is a maximal Cohen-Macaulay sheaf on $Y$, hence $\B_0$ is a maximal Cohen-Macaulay sheaf on
$Y^{sreg}$.
%Over $Y^{reg}$, $\B_0$
%is a vector bundle. So we need to show that $\B_0|_{Y_j}$ is maximal Cohen-Macaulay. First of all,
%we note that $\B_0|_{Y_j}=\hat{\iota}_{j*}\left((\B'_\hbar/\hbar\B'_\hbar)|_{Y_j^\times}\right)$.
%The proof of this equality is similar to the proof of Step 5 of Lemma \ref{Lem:HC_extension}.
%Now note that $\B_0|_{Y_j^\times}$ is an $\mathcal{O}$-coherent D-module on $Y_j^\times$. Also, thanks
%to Lemma \ref{Lem:compl_iso1}, this twisted D-module comes with an Euler derivation that
%then can be extended to a weakly $\C^\times$-equivariant structure. By \cite{sraco}
%-- see Proposition 3.5.3 there for a version, where instead of $Y_j^\times$ we deal with
%open subsets in affine spaces, for $Y_j^\times$ the proof is similar -- $\B_0|_{Y_j^\times}$
%is the direct sum of the isotypic components in the push-forward of $\Str_{\C^{2\wedge}\setminus \{0\}}$
%to $Y_j^\times$. The push-forwards of all these components to $Y_j$ are maximal Cohen-Macaulay
%modules. So we see that $\B_0$ itself is a maximal Cohen-Macaulay module.

Now we can use \cite[Expose VIII, Cor. 2.3]{Grothendieck} (together with the standard exact
sequence relating $H^*(Y^{sreg},\bullet)$ to $H^*_{Y\setminus Y^{sreg}}(\bullet)$) to see that
$H^1(Y^{sreg},\B_0)$ is finitely generated.
\end{proof}

\section{Enhanced restriction functor}\label{S_enhanced_res}
In this section we partially generalize constructions from \cite{HC,sraco} of ``enhanced''
restriction functors. Namely, we are going to produce a full embedding
$\overline{\HC}(\A_\lambda)\hookrightarrow \C\Gamma\operatorname{-mod}$ of monoidal categories, where $\Gamma=\pi_1^{alg}(Y^{reg})$. This functor upgrades the usual restriction functor from \cite[Section 3.3]{perv} associated to the open leaf whose target category is $\operatorname{Vect}$.

The functor we need will be constructed as the compostion of two intermediate
functors. The first functor will be a full monoidal embedding $\overline{\HC}(\A_\lambda)
\hookrightarrow \HC^\Gamma(\tilde{\A}_{\lambda^1}^0)$ (the definition of the latter
category will be given in Section \ref{SS_loc_functor}). Then we will produce a
monoidal equivalence   $\HC^\Gamma(\tilde{\A}_{\lambda^1}^{0})\xrightarrow{\sim}
\C\Gamma\operatorname{-mod}$. In Section \ref{SS_dagger_properties} we establish
basic properties of the composite functor $\overline{\HC}(\A_\lambda)\hookrightarrow \C\Gamma\operatorname{-mod}$.

The last two sections contain developments that are very closely related to
to the enhanced restriction functor. In Section \ref{SS_dagger_conseq} we will first give an alternative
formulation of the extension criterium from Section \ref{S_extension} in terms
of the representations of the groups $\Gamma, \Gamma_i, i=1,\ldots,k$.
Second, let $\Gamma'\subset \Gamma$ be a normal subgroup and $\hat{Y}$ be
the cover of $Y$ corresponding to $\Gamma/\Gamma'$. Let $\hat{\A}$ be a
quantization of $\hat{Y}$ with a $\Gamma/\Gamma'$-action. We will relate the categories $\overline{\HC}(\hat{\A})$
and $\overline{\HC}(\hat{\A}^{\Gamma/\Gamma'})$. These two results play a crucial
role in describing $\overline{\HC}(\A_\lambda)$. Finally, in Section
\ref{SS_trans_equiv} we discuss translation equivalences between the categories
$\overline{\HC}$ for different parameters and show that these equivalences
intertwine the enhanced restriction functors.

\subsection{Regluing quantizations}
In this section we are going to relate   quantizations of $Y^{reg}$ that have the
same period lying in $H^2(Y^{reg},\C)$. More precisely, we are going to show that any two such graded formal quantizations
$\D_h^1,\D_h^2$ are obtained from one another by gluing with respect to a 1-cocycle
of ``almost inner'' automorphisms. This is a partial generalization of a regluing result
from \cite[Section 2.5]{sraco}.

Let us cover $Y^{reg}$
with affine open $\C^\times$-stable subsets $U_{i}$.
Since $\D_h^1,\D_h^2$ have the same period and the period is functorial, we
see that $\D_h^1|_{U_i},\D_h^2|_{U_i}$ have the same period. By Corollary
\ref{Cor:graded_filt_quant}, $\D_h^1|_{U_i}\cong \D_h^2|_{U_i}$, a
$\C^\times$-equivariant isomorphism.  So $\D_h^2$ is obtained from $\D_h^1$
via regluing by a $1$-cocycle of automorphisms
$\theta_{ij}\in \operatorname{Aut}(\D^1_h|_{U_i\cap U_j})$
such that $\theta_{ij}$ is the identity modulo $h$.

The following claim is the main result of this section.

\begin{Prop}\label{Prop:derivation_inner}
There are  elements $f_{ij}\in \D_h^1|_{U_{ij}}$ that
\begin{itemize}
\item are $\C^\times$-invariant
\item
$\theta_{ij}=\exp(\operatorname{ad}(f_{ij}))$ and $f_{ij}=-f_{ji}$ for all $i,j$,
\item and  $f_{ik}=\log(\exp(f_{ij})\exp(f_{jk}))$ for all $i,j,k$.
\end{itemize}
\end{Prop}
\begin{proof}
Let $\delta_{ij},\alpha_{ij}$ have the same meaning as in the discussion preceding
Lemma \ref{Lem:period_regluing}.

Since the periods of $\D^1_h,\D^2_h$ coincide, it follows from Lemma
\ref{Lem:period_regluing} that $(\alpha_{ij})$ is a 1-coboundary: there are closed forms
$\alpha_{i}$ and functions $\underline{f}_{ij}$
with $\underline{f}_{ji}=-\underline{f}_{ij}, \alpha_{ij}=\alpha_i-\alpha_j+d\underline{f}_{ij}$ and
$\underline{f}_{ki}+\underline{f}_{ij}+\underline{f}_{jk}=0$. We can further assume that
the forms $\alpha_i$ and the functions $\underline{f}_{ij}$ are $\C^\times$-invariant.

Let $\underline{\delta}_{i}$ be the symplectic vector field on $U_i$ corresponding to $\alpha_i$.
By Lemma \ref{Lem:deriv_lift}, $\underline{\delta}_i$ lifts to a derivation $\delta_i$
of $\D^1_h|_{U_{i}}$. We can assume that $\delta_i$ has degree $-d$ with respect to the
$\C^\times$-action. So we can replace $\theta_{ij}$ with $\exp(-h\delta_i)\theta_{ij}\exp(h\delta_j)$
and assume that $\alpha_i=0$. Hence $\alpha_{ij}=d\underline{f}_{ij}$.

Let us prove  that there are $\C^\times$-invariant elements $f_{ij}\in D^1_h|_{U_{ij}}$ forming a
\v{C}ech cocycle such that $\delta_{ij}=h^{-1}\operatorname{ad}(f_{ij})$. This is true
modulo $h$ by the previous paragraph. Let $\underline{f}_{ij}'$ denote an arbitrary
$\C^\times$-invariant lift of $\underline{f}_{ij}$ to  $\D^1_h|_{U_{ij}}$. The derivation
$h^{-1}(\delta_{ij}-h^{-1}\operatorname{ad}(\underline{f}'_{ij}))$ has degree $-2d$.
The 1-form corresponding to this derivation modulo $h$ therefore has degree $-d$. It is
closed an hence exact. Arguing in this way, we see that there are  elements $f_{ij}$
with required properties.

It remains to prove that $f_{ik}=\log(\exp(f_{ij})\exp(f_{jk}))$.
Since $(\exp(\operatorname{ad}(f_{ij})))$ is a cocycle, the element $$z_{ijk}:=\log(\exp(f_{ki})\exp(f_{ij})\exp(f_{jk}))$$
is central in $\D^1_h|_{U_{ijk}}$. So it is a formal power series
$g_{ijk}(h)$. Since we know that $\underline{f}_{ij}+\underline{f}_{jk}+\underline{f}_{ki}=0$
we see that $g_{ijk}(h)$ is divisible by $h$. On the other hand, $z_{ijk}$ is $\C^\times$-invariant
hence constant. So it follows that it is zero and completes the proof of the proposition.
\end{proof}

\subsection{Embedding $\overline{\HC}(\A_\lambda)\hookrightarrow \HC^\Gamma(\tilde{\A}^{0}_{\lambda^1})$}\label{SS_loc_functor}
The goal of this section is to produce a full embedding $\HC(\A_\lambda)
\hookrightarrow \HC^\Gamma(\tilde{\A}_{\lambda^1}^0)$
of  monoidal categories. The notation here is as follows. As in Section
\ref{SS_symp_sing}, let
$\tilde{Y}^0$ denote the covering of $Y^{reg}$ with Galois
group $\Gamma$, this is an open subset in the conical symplectic singularity $\tilde{Y}$.
Note that we have an embedding $H^2(Y^{reg},\C)=H^2(\tilde{Y}^0,\C)^\Gamma
\hookrightarrow H^2(\tilde{Y}^0,\C)=H^2(\tilde{Y}^{reg},\C)$.
Abusing the notation, by $\lambda_0\in H^2(\tilde{Y}^0,\C)$ we denote the image of
$\lambda_0\in \h_0^*\subset H^2(Y^{reg},\C)$ under the embedding above.

Now let us define a quantization parameter $\lambda^1$ for $\tilde{Y}$.
Let $\mathcal{L}'_1,\ldots,\mathcal{L}'_\ell$ be the
codimension $2$ leaves of $\tilde{Y}$ and let $\Sigma'_i=\Disk^2/\Gamma'_i$
be the corresponding formal slices, $i=1,\ldots, \ell$. We will write
$W_{\tilde{Y},i},\h_{\tilde{Y},i}$ for the factor/summand in $W_{\tilde{Y}},\h_{\tilde{Y}}$
corresponding to the index $i$.

Let $\lambda^1_i, i=1,\ldots,\ell,$ denote the
quantization parameter for  $\Sigma'_i$ corresponding
to $1\in \C\Gamma'_i$. Note that this parameter is stable under the action
of $N_{\operatorname{SL}_2(\C)}(\Gamma_i')$ on $(\C\Gamma'_i)^{\Gamma'_i}_1$.
Hence $\lambda^1_i$ is stable under the monodromy action of
$\pi_1(\mathcal{L}'_i)$. We set $\lambda^1:=\lambda_0+\sum_{i=1}^\ell \lambda_i^1$,
this is an element of $\h^*_{\tilde{Y}}$.

So we have a filtered quantization $\tilde{\A}_{\lambda^1}$ of $\tilde{Y}$.

%Note that $\Gamma$ naturally acts on $\h_{\tilde{Y}}^*/W_{\tilde{Y}}$ as
%the latter space is an invariant of $\tilde{Y}$.

\begin{Lem}\label{Lem:Theta_inv}
The action of $\Gamma$ on $\C[\tilde{Y}]$ lifts to a $\Gamma$-action on
$\tilde{\A}_{\lambda^1}$.
\end{Lem}
\begin{proof}
Since $\Gamma$ acts on $\C[\tilde{Y}]$ by graded Poisson algebra
automorphisms, it also acts on the universal quantization $\tilde{\A}_{\h_{\tilde{Y}}}^{W_{\tilde{Y}}}$
of $\C[\tilde{Y}]$, see \cite[Section 3.7]{orbit}. We need to show that the parameter $W_{\tilde{Y}}\lambda^1$ is $\Gamma$-invariant.
The $H^2(\tilde{Y}^{reg},\C)$-component of $W_{\tilde{Y}}\lambda^1$ is $\Gamma$-stable by the construction. The group $\Gamma$ permutes the codimension $2$ symplectic leaves of
$\tilde{Y}$ hence induces a permutation of $\{1,\ldots,\ell\}$.   The element $\gamma$
gives rise to an isomorphism $\Sigma'_{i}\xrightarrow{\sim} \Sigma'_{\gamma(i)}$
compatible with the monodromy action.
Hence $\gamma$ gives rise to an isomorphism $\h^*_{\tilde{Y},i}/W_{\tilde{Y},i}
\xrightarrow{\sim} \h^*_{\tilde{Y},\gamma(i)}/W_{\tilde{Y},\gamma(i)}$.
This isomorphism must come from a diagram automorphism for the Dynkin diagram of $\Gamma'_{i}$.
Since $\lambda_i^1$ is invariant under the action of $N_{\operatorname{SL}_2(\C)}(\Gamma'_i)$,
we see $\gamma$ takes $W_{\tilde{Y},i}\lambda^1_i$ to
$W_{\tilde{Y},\gamma(i)}\lambda^1_{\gamma(i)}$. It follows that the parameter
$W_{\tilde{Y}}\lambda^1$ is $\Gamma$-invariant.
\end{proof}

Let $\tilde{\A}_{\lambda^1}^0$ denote the microlocalization of $\tilde{\A}_{\lambda^1}$
to $\tilde{Y}^0$.
HC $\tilde{\A}_{\lambda^1}^0$-bimodules were defined in Section
\ref{SS_Poisson_HC}.
By a $\Gamma$-equivariant
HC $\tilde{\A}_{\lambda^1}^0$-bimodule we mean a HC $\tilde{\A}_{\lambda^1}^0$-bimodule
$\B$ together with a $\Gamma$-action that is compatible with the actions of $\tilde{\A}_{\lambda^1}^0$.
The category of such HC bimodules will be denoted by $\HC^\Gamma(\tilde{\A}_{\lambda^1}^0)$.
This is a monoidal category.
%It comes with an internal Hom functor -- the sheaf Hom of left
%$\tilde{\A}_{\lambda^1}^0$-modules.

%Note that the category of HC bimodules $\tilde{\A}_{\lambda^1}^0$-bimodules is the quotient of
%the category of Poisson bimodules by the subcategory of torsion bimodules.

Now let us produce a full monoidal embedding $\overline{\HC}(\A_\lambda)\hookrightarrow
\HC^\Gamma(\tilde{\A}^0_{\lambda^1})$. The construction is based on  Proposition
\ref{Prop:derivation_inner} and follows the construction of an analogous functor
in \cite[Section 3.6]{sraco}.

First of all, we have the microlocalization functor $\overline{\HC}(\A_\lambda)
\rightarrow \HC(\A^{reg}_\lambda)$. This functor is a full
monoidal embedding.

The quantizations $\A^{reg}_\lambda, \left(\tilde{\A}^{0}_{\lambda^1}\right)^\Gamma$ have the same
period, equal to  $\lambda_0$.
So $\left(\tilde{\A}_{\lambda^1}^0\right)^\Gamma$ can be reglued from
$\A^{reg}_\lambda$ as explained in Proposition \ref{Prop:derivation_inner}. Similarly to \cite[Section 3.6]{sraco},
this yields a tensor category equivalence between $\HC(\A_\lambda^{reg})$ and
$\HC(\left(\tilde{\A}_{\lambda^1}^0\right)^\Gamma)$ by regluing. On the level of associated
graded bimodules, this equivalence is the identity.

Finally, as explained in Section \ref{SS_lift_etale}, the \'{e}tale morphism
$\tilde{Y}^0\rightarrow Y^{reg}$ gives rise to the pull-back functor between
the categories of Poisson bimodules. Note that the pull-back of a Poisson
$R_\hbar(\A_\lambda^{reg})$-bimodule has a natural $\Gamma$-equivariant structure.
Passing to the quotient categories by the $\hbar$-torsion bimodules
we get a functor $\HC(\left( \tilde{\A}_{\lambda^1}^0\right)^\Gamma)
\rightarrow \HC^\Gamma(\tilde{\A}^0_{\lambda^1})$. This is an equivalence
whose inverse is the push-forward functor followed by taking $\Gamma$-invariants.

Summarizing, we get a full monoidal embedding
$\mathsf{Loc}_{\tilde{Y}^0}:\overline{\HC}(\A_\lambda)\hookrightarrow \HC^\Gamma(\tilde{\A}^0_{\lambda^1})$ to be called the {\it localization functor} that is
the composition

$$\overline{\HC}(\A_\lambda)\hookrightarrow
\HC(\A_{\lambda}^{reg})\xrightarrow{\sim} \HC(\left(\tilde{\A}_{\lambda^1}^0\right)^\Gamma)
\xrightarrow{\sim} \HC^\Gamma(\tilde{\A}_{\lambda^1}^0).$$

%This composite embedding $\overline{\HC}(\A_\lambda)\hookrightarrow \HC^\Gamma(\tilde{\A}_{\lambda^1}^0)$
%will be called the {\it localization functor}.

\subsection{Equivalence $\HC^\Gamma(\tilde{\A}^{0}_{\lambda^1})\cong \C\Gamma\operatorname{-mod}$}\label{SS_HC_equiv}
Our goal now is to describe the monoidal category $\HC^\Gamma(\tilde{\A}^{0}_{\lambda^1})$.
First of all, we have a full embedding $\C\Gamma\operatorname{-mod}\hookrightarrow
\HC^\Gamma(\tilde{\A}^0_{\lambda^1}), V\mapsto V\otimes \tilde{\A}^0_{\lambda^1}$.
We want to prove that it is essentially surjective.

Our first step is the following lemma.

\begin{Lem}\label{Lem:HC_simpl_conn1}
We have $\HC(\tilde{\A}_{\lambda^1}^{0})\cong \operatorname{Vect}$
(with $\tilde{\A}_{\lambda^1}^0\in \HC(\tilde{\A}_{\lambda^1}^{0})$ corresponding to
$\C\in \operatorname{Vect}$).
\end{Lem}
\begin{proof}
%The proof is in two steps. We first show that $\overline{\HC}(\tilde{\A}_{\lambda^1})\cong \operatorname{Vect}$.
%Then we show that every HC $\tilde{\A}_{\lambda^1}^0$-bimodule extends to $\tilde{Y}$ using the extension
%criterium, Proposition \ref{Prop:HC_extension}.
%
%{\it Step 1}.

The proof is in several steps.

{\it Step 1}.
Let $\B\in \HC(\tilde{\A}_{\lambda^1}^{0})$. Pick a good filtration on $\B$ and let $\B_\hbar$
stand for the $\hbar$-adic completion of $R_\hbar(\B)$. This is a coherent Poisson $\tilde{\A}^0_{\lambda^1h}$-bimodule
that is flat over $\C[[\hbar]]$ (here, as before, $h=\hbar^d$).
Let $\iota^0:\tilde{Y}^0\hookrightarrow \tilde{Y}$ denote the inclusion. We claim that
\begin{equation}\label{eq:extension_prop}
\iota^0_*(\B_h)\text{ is coherent},\quad\iota^0_*(\B_h)/h \iota^0_*(\B_h)\cong
\C[\tilde{Y}]^{\oplus k}
\end{equation}
for some $k$.

{\it Step 2}.
The quotient $\B_\hbar/h \B_\hbar$ is a graded coherent Poisson $\Str_{\tilde{Y}^0}$-module.
So, by Lemma \ref{Lem:Poisson_classif}, it is isomorphic to $ \mathcal{O}_{\tilde{Y}^0}^{\oplus k}$. Let $Y'_j$ denote the spectrum of the completed local ring
of a point in the symplectic leaf $\mathcal{L}'_j\subset \tilde{Y}$. Let us show that
\begin{equation}\label{eq:ext_properties1}
H^0(Y'^\times_j, \B_\hbar|_{Y'^\times_j})|_{Y'^\times_j}\cong \B_\hbar|_{Y'^\times_j},\quad
H^0(Y'^\times_j, \B_\hbar|_{Y'^\times_j})/h
H^0(Y'^\times_j, \B_\hbar|_{Y'^\times_j})
\cong \Str_{Y'_j}^{\oplus k}.\end{equation}

Let us write $\Disk^{2n}$ for $\operatorname{Spec}(\C[[x_1,y_1,\ldots,x_n,y_n]])$,
the formal symplectic polydisk.
Let $\pi_j: \Disk^{2n}\setminus \Disk^{2n-2}
\twoheadrightarrow Y_j'^\times$ be the quotient morphism for the action of
$\Gamma'_j$.
Note that, by the construction of $\lambda^1$, $\tilde{\A}_{\lambda^1h}|_{Y_j'^\times}$
is the $\Gamma'_j$-invariants in the formal Weyl algebra of $\Disk^{2n}$ restricted
to  $\Disk^{2n}\setminus \Disk^{2n-2}$.

So $\pi_j^*(\B_{\hbar}|_{Y'^\times_j})$
is a coherent Poisson bimodule over that restriction. Such a bimodule is nothing else
as an $\Str[[\hbar]]$-coherent $D_{\Disk^{2n}\setminus \Disk^{2n-2}}[[\hbar]]$-module
(flat over $\C[[\hbar]]$). Its (D-module) pushforward  to $\Disk^{2n}$ is an $\Str[[\hbar]]$-coherent
$D_{\Disk^{2n}}[[\hbar]]$-module. So it is $\Str_{\Disk^{2n}}[[\hbar]]$.\footnote{An argument  along these lines has been communicated to me by Shilin Yu.} (\ref{eq:ext_properties1}) follows.
%Arguing similarly the proof of
%\cite[Proposition 3.5.4]{sraco} we see that $\pi_j^*(\B_{\hbar}|_{\tilde{Y}^\times_j})$ is
%a free $\Weyl_{h}|_{\Disk^{2n}\setminus \Disk^{2n-2}}$-module. This implies,
%in particular, that $\B_\hbar|_{\tilde{Y}^\times_j}$ extends to $\tilde{Y}_j$ and
%$$\hat{\iota}^0_{j*}(\B_\hbar|_{\tilde{Y}^\times_j})/h \hat{\iota}^0_{j*}(\B_\hbar|_{\tilde{Y}^\times_j})
%\xrightarrow{\sim} \hat{\iota}^0_{j*}(\B_\hbar|_{\tilde{Y}^\times_j}/h\B_\hbar|_{\tilde{Y}^\times_j}) %\cong \Str_{\tilde{Y}_j}^{\oplus k}.$$

{\it Step 3}.
Similarly, choose a smooth point $y$ in $\tilde{Y}^{reg}\setminus \tilde{Y}^0$. The complete analog
of (\ref{eq:ext_properties1}) holds at $y$. Now let $\iota'$ denote the inclusion $\tilde{Y}^0\hookrightarrow \tilde{Y}^{sreg}$. By Lemma \ref{Lem:HC_extension},
$\iota'_* \mathcal{B}_\hbar$ is a coherent Poisson bimodule. Moreover,
$\iota'_* \mathcal{B}_\hbar/h \iota'_* \mathcal{B}_\hbar$ is a vector bundle.
It follows that the analog of (\ref{eq:extension_prop}) holds for
$\iota'$ instead of $\iota^0$. Then we argue as in the proof of
Lemma \ref{Lem:HC_extension_codim4} to establish (\ref{eq:extension_prop}).

{\it Step 4}.
In particular, the right $\tilde{\A}_{\lambda^1h}$-module $H^0(\tilde{Y}^0,\B_\hbar)_{fin}$
is a graded deformation of $\C[\tilde{Y}]^{\oplus k}$ (where different summands can come with different
degrees).
Such a deformation is unique up to an  isomorphism of graded right $\tilde{\A}_{\lambda^1h}$-modules.
So $H^0(\tilde{Y}^0,\B_\hbar)$ is a free right $\tilde{\A}_{\lambda^1h}$-module.
To give a bimodule structure, i.e. a commuting $\tilde{\A}_{\lambda^1 h}$-action, on $H^0(\tilde{Y}^0,\B_\hbar)$
amounts to giving an algebra homomorphism $\varphi:\tilde{\A}_{\lambda^1h}
\rightarrow \operatorname{Mat}_k(\tilde{\A}_{\lambda^1h})$.
This homomorphism is the unit mod $h$. So it has the form $\operatorname{id}+h \delta+\ldots$,
where  $\delta \operatorname{ mod }h$ is a matrix $(\delta_{ij})_{i,j=1}^k$, where $\delta_{ij}$ is a Poisson derivation of
$\C[\tilde{Y}]$. By \cite[Proposition 2.14]{orbit}, $\delta_{ij}$ is inner, $\delta_{ij}=\{\underline{f}_{ij},\cdot\}$ for $\underline{f}_{ij}\in \C[\tilde{Y}]$.
We lift $\underline{f}_{ij}$ to $f_{ij}\in \tilde{\A}_{\lambda^1h}$ and form a matrix
$F:=(f_{ij})$. We have $\varphi=\operatorname{id}+[F,\cdot]+\ldots$, where $\ldots$
denotes elements starting with $h^2$. Then $\varphi\operatorname{Ad}(\exp(-F))-\operatorname{id}$
starts in degree at least 2 with respect to $h$. From here  and an easy induction
(on the smallest degree of $h$) we deduce that $\varphi=\operatorname{Ad}(\exp(\tilde{F}))$
for some $\tilde{F}\in \operatorname{Mat}_k(\tilde{\A}_{\lambda^1 h})$.
So $m\mapsto \exp(\tilde{A})m$ defines an isomorphism $\iota^0_{*}(\B_\hbar)\xrightarrow{\sim} \tilde{\A}_{\lambda^1 h}^{\oplus k}$ of  graded $\tilde{\A}_{\lambda^1 h}$-bimodules.  This finally implies the claim of the lemma.
%% and so has the form $\exp(\hbar \delta)$
%%for some derivation $\delta$. By \cite[Lemma 2.15]{orbit}, $\delta$ is inner. It follows that
%%$\iota^0_{*}(\B_\hbar)\cong \tilde{\A}_{\lambda^1\hbar^d}$.
\end{proof}

\begin{Cor}\label{Cor:HC_simply_conn}
The embedding $\C\Gamma\operatorname{-mod}\hookrightarrow \HC^\Gamma(\tilde{\A}_{\lambda^1}^{0})$
is an equivalence.
\end{Cor}
\begin{proof}
The left inverse is given by $\B$ going to the centralizer of $\tilde{\A}_{\lambda^1}$ in $\B$.
By Lemma \ref{Lem:HC_simpl_conn1} this is also the right inverse.
\end{proof}

So we get a  monoidal full embedding $\bullet_{\dagger}:
\overline{\HC}(\A_\lambda)\hookrightarrow \C\Gamma\operatorname{-mod}$ defined by $\B_{\dagger}\otimes \tilde{\A}_{\lambda^1}^0\cong \mathsf{Loc}_{\tilde{Y}^0}(\B)$.

\subsection{Properties of $\bullet_{\dagger}$}\label{SS_dagger_properties}
By the construction, one can also recover $\B_{\dagger}$ from a good filtration on $\B$.

\begin{Cor}\label{Cor:dagger_filt}
Let $\B\in \HC(\A_\lambda)$. Pick a good filtration on $\B$. Then the Poisson $\Str_{Y^{reg}}$-module $\gr\B|_{Y^{reg}}$
is obtained from the $\Gamma$-equivariant Poisson $\Str_{\tilde{Y}^0}$-module $\B_\dagger\otimes \Str_{\tilde{Y}^0}$ by equivariant descent.
\end{Cor}

In particular, this corollary shows that $\bullet_{\dagger}$ is independent of the regluing
elements $f_{ij}$ from Proposition \ref{Prop:derivation_inner}.

The next property of $\bullet_{\dagger}$ we will need is the existence of a right adjoint
functor. This functor will be denoted by $\bullet^{\dagger}$. Namely, let $\psi$ denote the equivalence
$\HC(\A_{\lambda}^{reg})\xrightarrow{\sim}\C\Gamma\operatorname{-mod}$
established in Sections \ref{SS_loc_functor}, \ref{SS_HC_equiv}. Then $V^\dagger:=H^0(Y^{reg}, \psi^{-1}(V))$.

The following properties are proved similarly to the analogous properties established
in \cite[Sections 3.4,4]{HC} and  \cite[Section 3.7]{sraco}.

\begin{Lem}\label{Lem:dagger_properties}
The functor $\bullet_{\dagger}$ has the following properties:
\begin{enumerate}
\item The kernel and the cokernel of the adjunction unit homomorphism $\B\rightarrow (\B_\dagger)^{\dagger}$
have associated varieties inside $Y^{sing}$.
\item The image of $\bullet_{\dagger}$ is closed under taking direct summands.
\end{enumerate}
\end{Lem}

Property (2) and the claim that $\bullet_{\dagger}$ is a monoidal functor imply that $\bullet_{\dagger}$ identifies $\overline{\HC}(\A_\lambda)$
with $\C(\Gamma/\Gamma_\lambda)\operatorname{-mod}$ for a uniquely determined
normal subgroup $\Gamma_\lambda\subset \Gamma$. This subgroup is recovered as the intersection
of the annihilators of the representations in  $\operatorname{im}\bullet_{\dagger}$.

\subsection{Consequences}\label{SS_dagger_conseq}
First, we want to give an equivalent formulation of Proposition \ref{Prop:HC_extension}
in terms of the functor $\bullet_{\dagger}$. Let us write $\underline{\A}^i_{\lambda_i}$
for the quantization of $\C^2/\Gamma_i$ corresponding to the parameter $\lambda_i$.
Recall that, for each $i$, we have natural homomorphism $\Gamma_i\rightarrow \Gamma$,
denote it by $\varphi_i$.

\begin{Prop}\label{Prop:extension_equiv}
Let $V\in \operatorname{Rep}\Gamma$. Then the following two claims are equivalent
\begin{itemize}
\item[(i)]
$V$
lies in the image of $\overline{\HC}(\A_\lambda)$
under $\bullet_{\dagger}$.
\item[(ii)] For all $i=1,\ldots,k$, the pullback
$\varphi_i^*(V)$ lies in the image of the functor
$\bullet_{\dagger,\Gamma_i}:\overline{\HC}(\underline{\A}^i_{\lambda_i})
\hookrightarrow \C\Gamma_i\operatorname{-mod}$.
\end{itemize}
\end{Prop}
\begin{proof}
We start with (i)$\Rightarrow$(ii). Let $\B\in \HC(\A_{\lambda})$ be such that
$\B_{\dagger}=V$. Recall that we have a functor $\bullet_{\dagger,i}:\HC(\A_\lambda)
\rightarrow \HC(\underline{\A}^i_{\lambda_i})$. We claim that
\begin{equation}\label{eq:dagger_transit}\left( \B_{\dagger,i}\right)_{\dagger,\Gamma_i}=
\varphi_i^*(V).\end{equation} This will imply (ii).

Let us prove (\ref{eq:dagger_transit}). Pick a good filtration on $\B$. This induces a
good filtration on $\B_{\dagger,i}$. Recall that $\Sigma_i$ is the formal slice to
$\mathcal{L}_i$, so that $\Sigma_i\cong \Disk^2/\Gamma_i$. The restriction of $\gr \B_{\dagger,i}$ to
$(\Disk^2/\Gamma_i)^\times$ coincides with the restriction of $\gr\B$ to $\Sigma_i^\times$
by the construction of $\bullet_{\dagger,i}$. On the other hand,  by Corollary
\ref{Cor:dagger_filt}, the restriction of
$\gr\B$ to $Y^{reg}$ is $\pi_*(V\otimes \Str_{\tilde{Y}_0})^\Gamma$ and, similarly,
the restriction of $\gr\B_{\dagger,i}$ to $(\C^2/\Gamma_i)^\times$ is
$\pi_{i*}(\left( \B_{\dagger,i}\right)_{\dagger,\Gamma_i}\otimes \Str_{\C^{2\times}})^{\Gamma_i}$.
But the restrictions of   $\gr\B|_{Y^{reg}}$ and $\gr\B_{\dagger,i}|_{(\C^2/\Gamma_i)^\times}$
to $(\Disk^2/\Gamma_i)^\times$ coincide. The homomorphism $\varphi_i:\Gamma_i\rightarrow \Gamma$
is the natural homomorphism $\pi_1^{alg}((\Disk^2/\Gamma_i)^\times)
\rightarrow \pi_1^{alg}(Y^{reg})$.  (\ref{eq:dagger_transit}) follows.

Now we prove (ii)$\Rightarrow$(i). For this we use
Proposition \ref{Prop:HC_extension}. Namely, consider $\B'\in \HC(\A_\lambda^{reg})$
corresponding to $V$ under the equivalence $\HC(\A_{\lambda}^{reg})
\cong \C\Gamma\operatorname{-mod}$. This HC bimodule comes with a natural good
filtration, let $\B'_\hbar$ be the $\hbar$-adic completion of the Rees bimodule of
$\B'$. Now consider the restriction $\B'_\hbar|_{Y_i^\times}$.

We claim that for all $i$ we have
\begin{itemize}
\item[(*)]
the restriction of $H^0(Y_i^\times, \B'_\hbar|_{Y_i^\times})$
to $Y_i^\times$ is $\B'_\hbar|_{Y_i^\times}$.
\end{itemize}
Once this is known, the proof of $V\in \operatorname{im}(\bullet_{\dagger})$
follows from Proposition  \ref{Prop:HC_extension}.

In the proof of (*) our first goal is
to extend $\B'_\hbar|_{Y_i^\times}$ to $(\C^{2n}\setminus \C^{2n-2})/\Gamma_i$.
Thanks to Lemma  \ref{Lem:compl_iso1}, we can equip $\B'_{\hbar}$ with an Euler derivation.
Consider the $R_\hbar^\wedge(\A_\lambda)|_{Y_i}$-module $H^0(Y_i^\times, \B'_\hbar|_{Y_i^\times}/ \hbar^m\B'_\hbar|_{Y_i^\times})$. It inherits the Euler derivation
from $\B'_\hbar|_{Y_i^\times}$, denote it also by $\mathsf{eu}$. Choose a $\Z$-equivariant map $\alpha:\C\rightarrow \Z$.
Define a $\C^\times$-action on the $\mathsf{eu}$-finite part of
$H^0(Y_i^\times, \B'_\hbar|_{Y_i^\times}/ \hbar^m\B'_\hbar|_{Y_i^\times})$
by declaring that $t\in \C^\times$ acts by $t^{\alpha(z)}$ on the generalized eigenspace for
$\operatorname{eu}$ with  eigenvalue $z$. This extends to a pro-rational $\C^\times$-action on
$H^0(Y_i^\times, \B'_\hbar|_{Y_i^\times}/ \hbar^m\B'_\hbar|_{Y_i^\times})$.
Consider the $\C^\times$-finite part $H^0(Y_i^\times, \B'_\hbar|_{Y_i^\times}/ \hbar^m\B'_\hbar|_{Y_i^\times})_{fin}$. This is a module over
$R_\hbar(\mathbb{A}\otimes\underline{\A}^i_{\lambda_i})/(\hbar^m)$. The pullback to $Y_i^\times$ of
$$H^0(Y_i^\times, \B'_\hbar|_{Y_i^\times}/ \hbar^m\B'_\hbar|_{Y_i^\times})/
\hbar^{m-1}H^0(Y_i^\times, \B'_\hbar|_{Y_i^\times}/ \hbar^m\B'_\hbar|_{Y_i^\times})
\rightarrow H^0(Y_i^\times, \B'_\hbar|_{Y_i^\times}/ \hbar^{m-1}\B'_\hbar|_{Y_i^\times})$$
is an isomorphism. It follows that the kernel and the cokernel of this homomorphism
are supported on $Y_i\cap \mathcal{L}_i$.
Therefore  after passing to $\C^\times$-finite part we have that the kernel and the cokernel
are still supported on the closed leaf.
Let $(\B'_\hbar|_{Y_i^\times}/\hbar^m \B'_\hbar|_{Y_i^\times})_{fin}$ be the microlocalization of  $H^0(Y_i^\times, \B'_\hbar|_{Y_i^\times}/ \hbar^m\B'_\hbar|_{Y_i^\times})_{fin}$ to $(\C^{2n}\setminus \C^{2n-2})/\Gamma_i$. We set $\B'_{\hbar,i,ext}:=\varprojlim (\B'_\hbar|_{Y_i^\times}/\hbar^m \B'_\hbar|_{Y_i^\times})_{fin}$. This is
a graded coherent Poisson module over the microlocalization of the $\hbar$-adic completion of
$R_\hbar(\mathbb{A}\otimes\underline{\A}^i_{\lambda_i})$. Its pullback to
$Y_i^\times$ is $\B'_\hbar|_{Y_i^\times}$. So $\B'_{\hbar,i,ext}$ is an extension we want.

Let $\B'_{\hbar,i,fin}$ denote the $\C^\times$-finite part of $\B'_{\hbar,i,ext}$.
Set $\B'_{i}:=\B'_{\hbar,i,fin}/(\hbar-1)\B'_{\hbar,i,fin}$. This is an
object of $$\HC([\mathbb{A}\otimes\underline{\A}^i_{\lambda_i}]|_{(\C^{2n}\setminus \C^{2n-2})/\Gamma_i}).$$
By the construction, its image under the equivalence of the latter category with
$\operatorname{Rep}(\Gamma_i)$ is $\varphi_i^*(V)$. It follows that
$\B'_{i}$ extends to an object of $\HC(\mathbb{A}\otimes\underline{\A}^i_{\lambda_i})$.
This implies (*) and hence finishes the proof.
%This is a
%$\A_{\lambda h}|_{Y_i^\times}$-bimodule corresponding to the $\Gamma_i$-module
%$\varphi_i^*(V)$. In particular, it extends to $Y_i$. Now we can apply Proposition %\ref{Prop:HC_extension}
%to see that $\B'_\hbar$ extends to $Y$, equivalently, $\B'$ is in the image of
%$\HC(\A_\lambda)$. And so $V$ is in the image of $\bullet_{\dagger}$.
\end{proof}

Now we proceed to the second part of this section, where we compare the categories
$\overline{\HC}(\hat{\A})$ and $\overline{\HC}(\A)$, where $\hat{\A}$
is a filtered quantization of $\hat{Y}$ and $\A:=\hat{\A}^{\Gamma/\Gamma'}$.
Here $\hat{Y}:=\operatorname{Spec}(\C[\hat{Y}^0])$ for a finite etale
cover $\hat{Y}^0$ of $Y^{reg}$ with Galois group $\Gamma/\Gamma'$.

\begin{Prop}\label{Prop:HC_inv}
Let $V$ be a $\C\Gamma$-module and let $V'$ be its restriction to $\Gamma'$.
In the notation above, the following two conditions are equivalent.
\begin{enumerate}
\item $V$ lies in the image of $\overline{\HC}(\A)$ under $\bullet_{\dagger}$.
\item $V'$ lies in the image of $\overline{\HC}(\hat{\A})$ under $\bullet_{\dagger'}$
(this is our notation for the $\bullet_{\dagger}$ for $\hat{Y}$).
\end{enumerate}
\end{Prop}
\begin{proof}
Let $\bullet^{\dagger'}$ denote the right adjoint for $\bullet_{\dagger'}$.
We claim that there is a natural $\Gamma/\Gamma'$-action on $V^{\dagger'}$ such that
\begin{equation}\label{eq:dagger_inv1} V^{\dagger}\cong (V^{\dagger'})^{\Gamma/\Gamma'}.
\end{equation}

We choose an affine covering $U_{i}$ of $Y^{reg}$. Let  $\pi':\hat{Y}^0\twoheadrightarrow Y^{reg}$ be the
quotient morphism for the $\Gamma/\Gamma'$-action on $\hat{Y}^0$. Set
$U'_i:=\pi'^{-1}(U_i)$.

Consider the quantization $\tilde{\A}:=\tilde{\A}_{\lambda^1}$, where $\lambda^1$
is constructed from $\lambda$, a quantization parameter for $\A$. Let $f_{ij}$
be the elements used to reglue $\A|_{Y^{reg}}$ to $\tilde{\A}^{\Gamma}|_{Y^{reg}}$.
Their pull-backs  $\pi'^*(f_{ij})$ are then used to reglue $\hat{\A}|_{\hat{Y}^0}$
to $\tilde{\A}^{\Gamma'}|_{\hat{Y}^0}$. This gives rise to an equivalence
$\psi':\HC^{\Gamma/\Gamma'}(\hat{\A}|_{\hat{Y}^0})\xrightarrow{\sim}\C\Gamma\operatorname{-mod}$.
The right adjoint of the resulting functor $\HC^{\Gamma/\Gamma'}(\hat{\A})\rightarrow
\C\Gamma\operatorname{-mod}$ is given by $V\mapsto H^0(\hat{Y}^0,\psi'^{-1}(V))$. The object
$V^{\dagger'}\in \HC(\hat{\A})$ is obtained from $H^0(\hat{Y}^0,\psi'^{-1}(V))\in
\HC^{\Gamma/\Gamma'}(\hat{\A})$ by forgetting the
action of $\Gamma/\Gamma'$. So $\Gamma/\Gamma'$ acts on $V^{\dagger'}$ and $V^{\dagger}=(V^{\dagger'})^{\Gamma/\Gamma'}$.

Also note that for $\B\in \HC^{\Gamma/\Gamma'}(\hat{\A})$ we have
a natural identification
\begin{equation}\label{eq:dagger_inv2} \B_{\dagger'}=(\B^{\Gamma/\Gamma'})_{\dagger}\end{equation}

Now we are ready to prove that  conditions (1) and (2) in the statement of the proposition
are equivalent.
Recall that, by Lemma \ref{Lem:dagger_properties},
$\bullet^{\dagger}:\C\Gamma\operatorname{-mod}\rightarrow \overline{\HC}(\A_\lambda)$
is left inverse to $\bullet_{\dagger}$. So $V$ lies in the image of $\bullet_{\dagger}$
if and only if $V=(V^\dagger)_\dagger$. Similarly, $V$ lies in the image of
$\bullet_{\dagger'}$ if and only if $V=(V^{\dagger'})_{\dagger'}$. Using (\ref{eq:dagger_inv1}) and (\ref{eq:dagger_inv2}), we see that $(V^{\dagger'})_{\dagger'}\cong (V^\dagger)_\dagger$.
The equivalence (1)$\Leftrightarrow$(2) follows.
\end{proof}

\subsection{Translation equivalences}\label{SS_trans_equiv}
Let $\lambda\in \h_Y^*$. It turns out that for certain values of $\lambda'$
the image of $\overline{\HC}(\A_{\lambda'})$ in $\C\Gamma\operatorname{-mod}$
coincides with the image of $\overline{\HC}(\A_{\lambda})$.

Namely, let us define the ``weight lattice'' $\Lambda_Y\subset \h_Y^*$. By definition,
it is the image of $\operatorname{Pic}(X^{reg})$ in $\h_Y^*$, where $X$ is a
$\Q$-terminalization of $Y$.
%Equivalently, $\Lambda$ is the direct sum
%$\bigoplus_{i=0}^k \Lambda_i\subset \bigoplus_{i=0}^k \h_i^*$,
%where the lattices $\Lambda_i$ are as follows: for $i=1,\ldots,k,$
%we write $\Lambda_i$ for the weight lattice in $\h_i^*$ and $\Lambda_0$
%is the image of $\operatorname{Pic}(Y^{reg})$ in $\h_0^*$.

Then we can form the ``extended affine Weyl group'' $W^{ae}_Y:=W_Y\ltimes \Lambda_Y$.

\begin{Lem}\label{Lem:trans_equiv}
Let $\lambda'\in W^{ae}_Y\lambda$. Then there is an equivalence
$\overline{\HC}(\A_\lambda)\xrightarrow{\sim} \overline{\HC}(\A_{\lambda'})$
intertwining the inclusions $\overline{\HC}(\A_\lambda),\overline{\HC}(\A_{\lambda'})
\hookrightarrow \C\Gamma\operatorname{-mod}$.
\end{Lem}
\begin{proof}
For $\lambda'\in W_Y\lambda$, the algebras $\A_\lambda,\A_{\lambda'}$ are the same
and the claim of the lemma follows. So it remains to consider the situation
when $\lambda'-\lambda\in \Lambda_Y$.

We can speak about HC $\A_{\lambda'}$-$\A_{\lambda}$-bimodules. Here is an example.
Lift $\lambda'-\lambda$ to an element $\chi\in \operatorname{Pic}(X^{reg})$.
Consider the line bundle $\mathcal{O}(\chi)$ on $X^{reg}$. Since $H^i(X^{reg},\Str)=0$
for $i=1,2$, we see that $\mathcal{O}(\chi)$ admits a unique filtered deformation
to a $\D_{\lambda'}^{reg}$-$\D_\lambda^{reg}$-bimodule, compare with
\cite[Section 5.1]{BPW}. The deformed bimodules
will be  denoted by $\D_{\lambda,\chi}^{reg}$.
Set $\A_{\lambda,\chi}:=\Gamma(\D_{\lambda,\chi})$. This is a HC $\A_{\lambda'}$-$\A_\lambda$-bimodule.
Similarly, we can consider the HC $\A_{\lambda}$-$\A_{\lambda'}$-bimodule
$\A_{\lambda',-\chi}$. The restrictions $\A_{\lambda,\chi}^{reg},\quad \A_{\lambda',-\chi}^{reg}$
to $Y^{reg}$ are mutually inverse. It follows that the functors $$\A_{\lambda,\chi}\otimes_{\A_\lambda}\bullet\otimes_{\A_\lambda}\A_{\lambda',-\chi},
\A_{\lambda',-\chi}\otimes_{\A_{\lambda'}}\bullet\otimes_{\A_{\lambda'}}\A_{\lambda,\chi}$$
  are mutually inverse equivalences $\overline{\HC}(\A_\lambda)\leftrightarrows \overline{\HC}(\A_{\lambda'})$.

It remains to show that these equivalences intertwine the embeddings
$\overline{\HC}(\A_\lambda),\overline{\HC}(\A_{\lambda'})\hookrightarrow
\C\Gamma\operatorname{-mod}$. To check this, for $\B\in \HC(\A_\lambda)$, we need to establish a good
filtration on
\begin{equation}\label{eq:tens_prod_bimod} \A_{\lambda,\chi}\otimes_{\A_\lambda}\B\otimes_{\A_\lambda}\A_{\lambda',-\chi}\end{equation}
in a natural way such that the restriction of its associated graded to $Y^{reg}$
is naturally identified with $\gr\B|_{Y^{reg}}$. For this we take the natural filtration
of the tensor product bimodule on (\ref{eq:tens_prod_bimod}). Since $\gr\A_{\lambda,\chi}|_{Y^{reg}}\cong \mathcal{O}(\chi)|_{Y^{reg}}$
and $\gr\A_{\lambda',-\chi}|_{Y^{reg}}\cong \mathcal{O}(-\chi)|_{Y^{reg}}$ are invertible
we see that
\begin{equation*}%\label{eq:assoc_gr_iso}
\gr\left(\A_{\lambda,\chi}\otimes_{\A_\lambda}\B\otimes_{\A_\lambda}\A_{\lambda',-\chi}\right)|_{Y^{reg}}
\cong \mathcal{O}(\chi)|_{Y^{reg}}\otimes \gr\B|_{Y^{reg}}\otimes \mathcal{O}(-\chi)|_{Y^{reg}}\cong
\gr \B|_{Y^{reg}}.
\end{equation*}
This finishes the proof of the lemma.
\end{proof}

\section{Classification for quantizations of Kleinian singularities}\label{S_classif_Klein}
The goal of this section is to prove a more precise version of Theorem \ref{Thm:Klein_classif_prelim}.
Recall that $\g,W_\Gamma,\h_\Gamma$ denote the Lie algebra, Weyl group and the Cartan space of the same type as $\Gamma$,
 $\Lambda_r\subset \h_\Gamma^*$ is the root lattice, and $W^a_\Gamma:=W_\Gamma\ltimes \Lambda_r$
 is the affine Weyl group.

\begin{Thm}\label{Thm:Klein_classif}
Let $\Gamma\subset \SL_2(\C)$ be a finite subgroup not of type $E_8$.
The following claims are true:
\begin{enumerate}
\item For each $c\in (\C\Gamma)_1^\Gamma$, there is a minimal normal subgroup $\Gamma_c\subset \Gamma$ such that $W^a_\Gamma \lambda_c$ contains $\lambda_{c'}$ with $c'\in \C\Gamma_c(\subset \C\Gamma)$.
\item Let $\lambda\in \h^*_\Gamma$ be the parameter corresponding to $c$. Then the image of $\overline{\HC}(\A_{\lambda})$ in $\C\Gamma\operatorname{-mod}$ under $\bullet_\dagger$ is $\C(\Gamma/\Gamma_c)\operatorname{-mod}$.
\end{enumerate}
\end{Thm}

The scheme of the proof of (2) is, essentially, as follows.
First, we show that a one-dimensional representation of $\Gamma$ lies in the image of $\bullet_{\dagger}$ if
and only if $\Gamma_c$ acts trivially on it. For this we use the known description of
$\overline{\HC}(\U_\lambda)$ (here $\U_\lambda$ is the central reduction of the universal
enveloping algebra $\U:=U(\g)$) and Proposition \ref{Prop:extension_equiv} that  allows us to relate
$\overline{\HC}(\U_\lambda)$ and $\overline{\HC}(\A_\lambda)$. To extend (2) to
higher dimensional irreducible representations of $\Gamma$ we use translation equivalences
and Proposition \ref{Prop:HC_inv}.

\subsection{One-dimensional representations in the image of $\bullet_{\dagger}$}\label{SS_classif_1dim}
Our first task is to describe the one-dimensional representations of $\Gamma$
lying in the image of $\bullet_{\dagger}$. The initial step is to recall the description
of the category $\overline{\HC}(\U_\lambda)$.

Below we write $\h,W,W^a$ for $\h_\Gamma,W_\Gamma,W^a_\Gamma$,
we view $\h$ as a Cartan subalgebra of $\g$.
Let $\Lambda$
denote the weight lattice in $\h^*$. Note that $(\Lambda/\Lambda_r)^*$ coincides
with $\Gamma_\g:=\pi_1(Y^{reg})$, where $Y$ stands for the nilpotent cone.
Form  the extended affine
Weyl group $W^{ae}:=W\ltimes \Lambda$ so that $W^a$ is a normal subgroup in $W^{ae}$
and $W^{ae}/W^a\cong \Lambda/\Lambda_r$.

Here is a description of $\overline{\HC}(\U_\lambda)$. It is standard but since we haven't
found a proof in the literature, we provide it.

\begin{Lem}\label{Lem:HC_nilp}
The image under $\bullet_{\dagger}$ of  $\overline{\HC}(\U_\lambda)$ in $\operatorname{Rep}[(\Lambda/\Lambda_r)^*]$
coincides with the subcategory of representations whose irreducible constituents lie in
$W^{ae}_\lambda/W^a_\lambda$ (naturally embedded into $W^{ae}/W^a$).
\end{Lem}
\begin{proof}
%For $\lambda'\in W^{ae}\lambda$, the category $\overline{\HC}(\U_{\lambda'},\U_{\lambda})$
%contains a translation bimodule that gives a required equivalence of right module categories.
%So to prove (1) we only need to check that $\overline{\HC}(\U_{\lambda'},\U_{\lambda})\neq \{0\}$
%implies $\lambda'\in W^{ae}\lambda$.
In the proof we can assume that $\lambda$ is regular. Indeed, if $\lambda'-\lambda\in \Lambda$,
then the images of $\overline{\HC}(\U_\lambda)$ and $\overline{\HC}(\U_{\lambda'})$ in $\C(\Lambda/\Lambda_r)^*\operatorname{-mod}$ are the same,
as was explained in Lemma \ref{Lem:trans_equiv} (in the case of general $Y$).

Note that to an irreducible HC $\U$-bimodule, $\B$, we can assign an element of
$\Lambda/\Lambda_r$ as follows: this is the $\Lambda_r$-coset of weights of $\g$ in its adjoint
(and hence locally finite) action on $\B$.  Since $\B$
is irreducible, we have $\B=\U b\U$ for any $b\in \B$. Hence  all the weights lie in a single coset.

Let us show that the  irreducible
$\Gamma_\g$-module corresponding to this element coincides with $\B_{\dagger}$.
Indeed, $\Gamma_\g$ is identified with the center $Z$ of the simply connected
algebraic group $G$ with Lie algebra $\g$. In our case -- when $Y$ is the nilpotent cone -- the functor $\bullet_{\dagger}$ constructed in the end of Section \ref{SS_HC_equiv}
 is a special case of the functor $\bullet_{\dagger}$ constructed
in \cite[Section 3.4]{HC} in the case of the principal orbit. The source of that functor is the category $\HC(\U)$ of all HC $\U$-bimodules
and the target is the category $\HC^Z(\mathcal{W})$ of $Z$-equivariant HC bimodules over
the W-algebra $\mathcal{W}$ constructed for the principal orbit. This W-algebra coincides with
the center of $\U$. By the construction of the functor in \cite{HC} it is clear that if $Z$ acts on $\B$
via a character $\chi$, then it acts on $\B_{\dagger}$ via that character as well. This shows the
claim on the coincidence of two $\Gamma_\g$-modules in the beginning of the paragraph.

So we need to understand the set $\{\lambda+\Lambda_r\}$, where $\lambda$ runs over
the possible weights of HC bimodules in $\overline{\HC}(\U_\lambda)$.  All classes from $W^{ae}_\lambda/W^a_\lambda$
are realized by translation bimodules that, for general $Y$, were introduced
in the proof of Lemma \ref{Lem:trans_equiv}. So we need to show that no other
classes  appear. All weights that appear in $\B\in \HC(\U_\lambda)$ also
appear in $\operatorname{pr}_\lambda(V\otimes \U_\lambda)$,
where $V$ is a finite dimensional $\g$-module and $\operatorname{pr}_\lambda$ stands for the projection
to the infinitesimal block with central character $\lambda$. Also note that for every
$\B\in \HC(\U_\lambda)$ there is a Verma module $\Delta(w\cdot \lambda)$ for $w\in W$ in the infinitesimal block $\mathcal{O}_\lambda$
of the BGG category $\mathcal{O}$ such that $\B\otimes_{\U_\lambda}\Delta(w\cdot\lambda)\neq 0$, this follows,
for example, from \cite{BG}.  So we reduce to showing that
all weights that appear in $\operatorname{pr}_\lambda(V\otimes \Delta(w\cdot \lambda))$
are of the form $w\cdot \lambda+\chi$ with $\chi+\Lambda_r\in W^{ae}_\lambda/ W^a_\lambda$.
On the other hand, any weight appearing in $\operatorname{pr}_\lambda(V\otimes \Delta(w\cdot \lambda))$
should appear in $\Delta(u\cdot \lambda)$ for some $u\in W$ hence lies in $W^a\lambda$. Clearly,
$w\cdot \lambda+\chi\in W^a\lambda$ is equivalent to $\chi+\Lambda_r\in W^{ae}_\lambda/W^a_\lambda$,
which is precisely what we need.
%Let $\B\in \HC(\U_{\lambda'},\U_\lambda)$. We have a functor $\HC(\U_{\lambda'},\U_\lambda)$ to the category %$\mathcal{O}_{\lambda'}$
%of modules over $\U_{\lambda'}$ (=the category of all finitely generated $\U_{\lambda'}$-modules
%with locally nilpotent action of $\mathfrak{n}$) given by $\B\mapsto \B\otimes_{\U_\lambda}\bigoplus_{w\in %W}\Delta(w\lambda)$, where $\Delta(w\lambda)$ stands for the Verma module with highest weight $w\lambda-\rho$..
%By results of Bernstein and Gelfand, this functor is faithful. On the other hand,
%$\B$ is a quotient of $V\otimes \U_{\lambda}$ for some finite dimensional $\g$-module $V$.
%It follows that $\lambda'$ is a central character that occurs in $V\otimes\bigoplus_{w\in W}\Delta(w\lambda)$.
%Therefore $\lambda'\in W^{ae}\lambda$. This completes the proof of (1).
%
%Now arguing as in the proof of (1),
%we see that $\chi+\Lambda$ occurs for a bimodule in $\HC(\U_\lambda)$ if and only if
%$\lambda+\chi$ is in $W^a\lambda$. It is easy to see that once $\chi$ occurs for some
%bimodule in $\HC(\U_\lambda)$, it also occurs for some bimodule in $\overline{\HC}(\U_\lambda)$.
\end{proof}

With Lemma \ref{Lem:HC_nilp} we can now use Proposition
\ref{Prop:extension_equiv} to describe the one-dimensional representations of $\Gamma$ in
the image of $\HC(\A_\lambda)$. For this, note that $\Gamma/(\Gamma,\Gamma)\cong \Gamma_\g$.
So the one-dimensional $\Gamma$-modules are in a one-to-one correspondence with $\Lambda/\Lambda_r$.

\begin{Prop}\label{Prop:HC_Klein_rk1}
Let $V$ be a one-dimensional $\Gamma$-module. Then the following claims are equivalent:
\begin{enumerate}
\item $V$ lies in $W^{ae}_\lambda/W^a_\lambda$.
\item $V$ lies in the image of $\HC(\A_\lambda)$ under $\bullet_{\dagger}$.
\end{enumerate}
\end{Prop}
\begin{proof}
The variety $Y$ has a unique symplectic leaf of codimension $2$ and the corresponding Kleinian group $\Gamma$
has the same type as $\g$. Note the  homomorphism $\Gamma\rightarrow \Gamma_\g$
of algebraic fundamental groups is an epimorphism (hence gives an isomorphism
$\Gamma/(\Gamma,\Gamma)\xrightarrow{\sim}\Gamma_\g$). Indeed, assume that $\Gamma\rightarrow \Gamma_{\g}$
is not surjective. Equivalently, there is a nontrivial irreducible $\Gamma_{\g}$-module, say $V$, with trivial
pull-back to $\Gamma$. By Proposition \ref{Prop:extension_equiv}, $V$ lies in the image of $\HC(\U_\lambda)$
for all $\lambda$. This contradicts Lemma \ref{Lem:HC_nilp}. The surjectivity can also be checked case by case.

According to Lemma \ref{Lem:HC_nilp},
(1) is equivalent to $V$ lying in the  image of $\overline{\HC}(\U_\lambda)$.
The latter condition is equivalent to (2), this is a special
case of Proposition \ref{Prop:extension_equiv}.
\end{proof}
%\subsection{Translation equivalence}

\subsection{Subgroup $\Gamma_c$}
In this section, to a parameter $\lambda\in \h_\Gamma^*$ or, equivalently, $c\in (\C\Gamma)^\Gamma_1$ we
assign a normal subgroup $\Gamma_c$. The construction will be inductive. Let $\Gamma'_{c}$ be the normal subgroup in $\Gamma$ that is the intersection of the kernels of the one-dimensional representations of $\Gamma$
that lie in $W^{ae}_\lambda/W^a_\lambda$. So $(\Gamma/\Gamma'_c)^*=W_{\lambda}^{ae}/W_\lambda^a$.

\begin{Lem}\label{Lem:conj_spec_param}
There is an element $\lambda'\in W^{a}\lambda$ such that the corresponding parameter $c'$
lies in $\C\Gamma'_{c}$.
\end{Lem}
\begin{proof}
In the proof of the lemma we can assume that $\lambda$ is real, i.e., $\lambda\in \mathbb{R}\otimes_{\Z}\Lambda$.
Indeed, the locus of $\lambda$ such that $c'\in \C\Gamma'_{c}$ is the union of affine subspaces of
$\h^*$ defined over $\mathbb{R}$. Similarly, the locus of parameters $\lambda$ with given group $W^{ae}_\lambda/W^a_\lambda$ is the union of affine subspaces of $\h^*$ defined over $\R$.

We claim that for $\lambda'\in W^{a}\lambda$ lying in the fundamental alcove we have $c'\in
\C\Gamma'_{c}$. We have $W^{ae}=(\Lambda/\Lambda_r)\ltimes W^a$ and $W^{ae}_{\lambda'}/W^{a}_{\lambda'}
=(\Lambda/\Lambda_r)_{\lambda'}$, where we view $\Lambda/\Lambda_r$ as a group acting on the
fundamental alcove. The action of $\Lambda/\Lambda_r$ comes from automorphisms of the affine Dynkin diagram.
Let $\mathfrak{A}_{\lambda'}$ be the group of the automorphisms of the affine Dynkin diagram
coming from $(\Lambda/\Lambda_r)_{\lambda'}$. The group $\Gamma'_{c}$ is the largest normal
subgroup  $\Gamma_1\subset\Gamma$ such that the $\Gamma$-irreducibles in the same $\mathfrak{A}_{\lambda'}$-orbit
become isomorphic over $\Gamma_1$. This follows from the observation that the action of
$(\Lambda/\Lambda_r)_{\lambda'}$ on the $\Gamma$-irreducibles is by tensoring with the one-dimensional
$\Gamma/\Gamma'_c$-modules.

Now a case by case analysis shows that $c'\in \C\Gamma'_{c}$ if $\lambda'$
lies in the fundamental alcove and $(\Lambda/\Lambda_r)^*_{\lambda'}=\Gamma/\Gamma'_{c}$.
\end{proof}

We define a sequence of normal in $\Gamma$ subgroups $\Gamma'_{c}\supset \Gamma''_{c}
\supset \ldots$ as follows. We set $\Gamma''_{c}:=(\Gamma'_{c'})'$.
Let $\h_1$ and $W_1$ be the Cartan space and the finite Weyl group for $\Gamma'_{c}$.
Inside $\h_1$ we have the subspace of $\Gamma/\Gamma'_{c}$-invariant elements, denote
it by $\underline{\h}_1$ (the action of $\Gamma/\Gamma'_{c}$ comes from twisting
the irreducible $\Gamma_{c'}$-modules so it is by diagram automorphisms).
Note that $\underline{\h}_1^*$ naturally embeds into $\h^*$ and under the affine 
identification $\h^*\cong (\C\Gamma)^\Gamma_1$, the space $\underline{\h}_1^*$
is identified with $(\C\Gamma_c')^\Gamma_1$. 

Set $\underline{W}_1:=W_1^{\Gamma/\Gamma'_c}$, this
group acts faithfully on  $\underline{\h}_1$. The simple roots for $(\underline{\h}_1,\underline{W}_1)$
are exactly the elements of the form $\sum_{\alpha\in \mathfrak{O}}\alpha$, where $\mathfrak{O}$ runs
over the set of orbits of  $\Gamma/\Gamma'_c$ on the set of simple roots for $W_1$.
The simple coroots have the same description. It follows that the root lattice $\underline{\Lambda}_{1}'$ for $\underline{W}_1$ coincides with the intersection  $\Lambda_{1}'\cap \underline{\h}_1^*$ (and the similar claim holds
for the coroot lattice). Also the fundamental chamber for $(\underline{W}_1,\underline{\h}_1)$
is the intersection of that for $(W_1,\h_1)$ with $\underline{\h}_1$.
The maximal roots and coroots for $\h_1$ and $\underline{\h}_1$
coincide. It follows that the fundamental alcove in $\underline{\h}_{1\R}^*$
coincides with the intersection of  $\underline{\h}_{1\R}^*$ and the fundamental alcove in $\h_{1\R}^*$.
In particular, we have the following analog of Lemma \ref{Lem:conj_spec_param} (note that if we replace
$\underline{W}_1^a$ with $W_1^a$ this is just Lemma \ref{Lem:conj_spec_param}).

\begin{Lem}\label{Lem:conj_spec_param1}
There is an element $\lambda''\in \underline{W}_1^{a}\lambda'$ such that the corresponding parameter $c''$
lies in $\C\Gamma''_{c}$.
\end{Lem}

We produce $\Gamma^{(k)}_{c}$ from $\Gamma^{(k-1)}_{c}$ in the similar way. The sequence $\Gamma^{(k)}_{c}$ clearly stabilizes. Let
$c^{(k)}\in \C\Gamma_{c}^{(k)}$ be a resulting parameter. Let $\lambda^{(k)}\in \h^*$ be
the  parameter corresponding to $c^{(k)}$.

\begin{Lem}\label{Lem:conj_spec_param2}
We have $\lambda^{(k)}\in W^a \lambda$.
\end{Lem}
\begin{proof}
The root system for $\underline{W}_1^a$ is obtained from that for $W^a$ as the invariants under the action of
a diagram automorphism group described in the proof of Lemma \ref{Lem:conj_spec_param}.
It follows that $\underline{W}_1^a$ is a subgroup of $W^a$ and it acts on $\underline{\h}_1^*$
as the subgroup. Therefore $\lambda''\in W^a\lambda$. Continuing to argue in the same way we
see that $\lambda^{(k)}\in W^a\lambda$.
\end{proof}

\subsection{Proof of Theorem \ref{Thm:Klein_classif}}\label{SS_Klein_classif}
The following proposition completes the proof of Theorem \ref{Thm:Klein_classif}.
Let us write $\Gamma_c$ for the stable normal subgroup $\Gamma^{(k)}_{c}$.

\begin{Prop}\label{Prop:HC_Klein1}
Suppose that $\Gamma$ is solvable, equivalently, not of type $E_8$. The subgroup $\Gamma_{c}\subset \Gamma$ is a
unique normal subgroup satisfying either of the following two properties:
\begin{enumerate}
\item The image of $\overline{\HC}(\A_\lambda)$ in $\C\Gamma\operatorname{-mod}$ is $\C(\Gamma/\Gamma_c)\operatorname{-mod}$.
\item $\Gamma_{c}$ is the minimal normal  subgroup $\Gamma_0$ of $\Gamma$ such that there is
a parameter $c^0\in \C\Gamma_0$ with the property that the corresponding parameter
$\lambda^0$ lies in $W^a\lambda$.
\end{enumerate}
\end{Prop}
\begin{proof}
Let us prove that the image of  $\overline{\HC}(\A_\lambda)$ in $\C\Gamma\operatorname{-mod}$
under $\bullet_{\dagger}$ coincides with $\C(\Gamma/\Gamma_{c})\operatorname{-mod}$ (which determines $\Gamma_{c}$
uniquely assuming such a subgroup exists). The proof is induction on the number of elements in $\Gamma$.

First of all, consider the situation when $\Gamma=\Gamma'_{c}$. We claim that in this case
the image of $\overline{\HC}(\A_\lambda)$ consists of the trivial representations.
To prove this recall that the image is a tensor subcategory. If it contains some irreducible
representation $V$ of $\Gamma$, then it contains all irreducible representations of $\Gamma/\Gamma_0$,
where $\Gamma_0$ is  the kernel of $\Gamma\rightarrow \operatorname{GL}(V)$. Since $\Gamma$ is solvable, we see that
$\Gamma/\Gamma_0$ has a nontrivial one-dimensional representation. Then $\Gamma\neq \Gamma'_{c}$
thanks to Proposition \ref{Prop:HC_Klein_rk1}. We get a contradiction.

Now we can assume, by induction that our claim is proved for $\Gamma_{c}'$ instead of $\Gamma$. Pick $\lambda'$
as in Lemma \ref{Lem:conj_spec_param}. Then we have an equivalence $\HC(\A_{\lambda})
\cong \HC(\A_{\lambda'})$ that intertwines the embeddings of these categories into $\C\Gamma\operatorname{-mod}$
by Lemma \ref{Lem:trans_equiv}. Now for $\lambda'$ instead of $\lambda$ our claim follows  from Proposition
\ref{Prop:HC_inv} because $\A_{\lambda'}$ is realized as
$\Gamma/\Gamma'_c$-invariants in the quantization of $\C^2/\Gamma'_c$ with
parameter $c'$. This finishes the first characterization of $\lambda$.

Let us prove the  characterization of $\Gamma_c$ in (2). Let $\Gamma_0$ be a normal subgroup such that
there is $\lambda^0\in W^a\lambda$ with $c^0\in \C\Gamma_0$. Then $\C(\Gamma/\Gamma_0)\operatorname{-mod}$
lies in the image of $\HC(\A_\lambda)$ under $\bullet_{\dagger}$. As we have seen, this image coincides with $\C(\Gamma/\Gamma_{c})\operatorname{-mod}$. So $\Gamma_{c}\subset \Gamma_0$. And,
by Lemma \ref{Lem:conj_spec_param2}, we know that $W^a\lambda$ intersects $(\C\Gamma_c)^\Gamma$.
\end{proof}

\section{Classification for quantizations of symplectic singularities}\label{S_classif_gener}
\subsection{Structure of $\HC(\A_\lambda)$}
In this section we consider a conical symplectic singularity $Y$ without
dimension two slices of type $E_8$. Let, as before, $\Gamma$ denote the algebraic
fundamental group of $Y^{reg}$. Let $\Gamma_i,i=1,\ldots,k,$ be the Kleinian
groups corresponding to codimension $2$ symplectic leaves so that we have
group homomorphisms $\varphi_i:\Gamma_i\rightarrow \Gamma$. Inside $\Gamma_i$
we have a normal subgroup to be denoted by $\Gamma_{i,\lambda}$. Namely,
let $\lambda_i$ be the component of $\lambda$ in $\h_i^*$. We can
produce $c_i\in \C\Gamma_i$ out of $\lambda_i$ and form the normal
subgroup $\Gamma_{i,c_i}\subset \Gamma_i$ as in Proposition
\ref{Prop:HC_Klein1}. To simplify the notation,
we write $\Gamma_{i,\lambda}$ for $\Gamma_{i,c_i}$. Finally,
let $\Gamma_{\lambda}$ denote the minimal normal subgroup of $\Gamma$
containing  $\varphi_i(\Gamma_{i,\lambda})$ for all $i$.

The following theorem is a direct corollary of Proposition \ref{Prop:extension_equiv}
combined with Theorem \ref{Thm:Klein_classif}. It strengthens Theorem 
\ref{Thm:classif_sympl_prelim}.

\begin{Thm}\label{Thm:HC_classif}
The functor $\bullet_{\dagger}$ identifies $\overline{\HC}(\A_\lambda)$
with $\C(\Gamma/\Gamma_\lambda)\operatorname{-mod}$.
\end{Thm}

\subsection{Towards description  of $\HC(\A_{\lambda'},\A_\lambda)$}\label{SS_classif_diff_param}
Let as before $Y$ be a conical symplectic singularity. Let $\lambda,\lambda'\in \h^*_Y$.
As was mentioned in the introduction, it makes sense to speak about HC $\A_{\lambda'}$-$\A_\lambda$-bimodules.
Let $\HC(\A_{\lambda'},\A_{\lambda})$ denote the category of such bimodules and
$\overline{\HC}(\A_{\lambda'},\A_{\lambda})$ denote the quotient by the full subcategory of  bimodules
with proper associated varieties.

One could ask to describe the category $\overline{\HC}(\A_{\lambda'},\A_\lambda)$
similarly to the description of $\overline{\HC}(\A_\lambda)$. The easiest case
is when $\h_0^*=\{0\}$, let us consider it first. Recall the weight lattice
$\Lambda_Y\subset \h_Y^*$, the image of $\operatorname{Pic}(X^{reg})$. Recall  the extended affine Weyl group $W^{ae}_Y:=W_Y\ltimes \Lambda_Y$.

\begin{Conj}\label{Conj:classif_diff_param1}
Suppose that $\h_0^*=\{0\}$. Then $\overline{\HC}(\A_{\lambda'},\A_{\lambda})\neq \{0\}$
if and only if $\lambda'\in W_Y^{ae} \lambda$. Moreover, if $\lambda'\in W^{ae}_Y \lambda$ the categories
$\overline{\HC}(\A_{\lambda'},\A_{\lambda}), \overline{\HC}(\A_{\lambda},\A_{\lambda'})$
contain mutually inverse objects. In particular, we have an equivalence $\overline{\HC}(\A_{\lambda'},\A_\lambda)
\cong \overline{\HC}(\A_\lambda)$ of right $\overline{\HC}(\A_\lambda)$-module categories.
\end{Conj}

Let us explain an approach to this conjecture, where we have not worked out some technical details.
It is easy to see that if $\lambda'\in W^{ae}_Y \lambda$, then the conclusions of the conjecture
hold (thanks to the translation bimodules introduced in Section \ref{SS_trans_equiv}). In fact,
here we do not need to require that $\h_0^*=\{0\}$, this restriction is only needed for the
opposite direction. So what remains to show is that $\overline{\HC}(\A_{\lambda'},\A_{\lambda})\neq \{0\}$
implies that $\lambda'\in W^a_Y\lambda$. Using Proposition \ref{Prop:HC_extension},
we reduce the proof to the case when $Y=\C^2/\Gamma$. In this case (and in the more general case of
symplectic quotient singularities), it was shown in  \cite[Section 3.5]{sraco} that $\overline{\HC}(\A_{\lambda'},\A_{\lambda})\neq \{0\}$ implies that
$\HC(\A_{\lambda'}^{reg},\A_\lambda^{reg})\cong \C\Gamma\operatorname{-mod}$ (an
equivalence of $\HC(\A_{\lambda'}^{reg})$-$\HC(\A_\lambda^{reg})$-bimodule categories).
Then we can classify the irreducibles in $\overline{\HC}(\A_{\lambda'},\A_\lambda)$
using techniques similar to Sections \ref{SS_classif_1dim} and \ref{SS_Klein_classif}
and the following easy observation: if $V$ lies in the image of $\overline{\HC}(\A_{\lambda'},\A_{\lambda})
\hookrightarrow \C\Gamma\operatorname{-mod}$, then $V^*$ lies in the image of
$\overline{\HC}(\A_{\lambda},\A_{\lambda'})$ and hence $V^*\otimes V$ and $V\otimes V^*$
lie in the images of $\overline{\HC}(\A_{\lambda})$ and $\overline{\HC}(\A_{\lambda'})$.

The case when $\h_0^*\neq \{0\}$ is more difficult. There are examples where $\overline{\HC}(\A_{\lambda'},\A_\lambda)\neq \{0\}$ while $\lambda'\not\in W^a_Y \lambda$ (e.g. some cases of $\operatorname{Spec}(\C[\Orb])$, where
$\operatorname{codim}_{\overline{\Orb}}\overline{\Orb}\setminus \Orb\geqslant 4$, this was studied
in \cite{quant_nilp}). A general reason for the complications is the difference between the groups $\operatorname{Pic}(\tilde{Y}^0)^{\Gamma}$ and $\operatorname{Pic}(Y^{reg})$ -- the latter is a finite
index subgroup in the former and may be proper. Because of this, for $\chi\in \operatorname{Pic}(\tilde{Y}^0)^{\Gamma}$
the translation bimodule $\tilde{\A}^0_{\lambda,\chi}$ carries only a projective representation of $\Gamma$.
One  should be able to prove that $\overline{\HC}(\A_{\lambda'},\A_{\lambda})\neq \{0\}$
implies that the $\h_0^*$-component $\lambda'-\lambda$ lies in the image of $\operatorname{Pic}(\tilde{Y}^0)^\Gamma$.
If that is the case, $\overline{\HC}(\A_{\lambda'},\A_{\lambda})$ should embed into the category
of projective $\Gamma$-representations with the Schur multiplier determined by $\lambda'-\lambda$
(and equal to zero if the $\h_0^*$-component of  $\lambda'-\lambda$ is in the image of $\operatorname{Pic}(Y^{reg})$).
After that it should not be difficult to prove a result similar to Theorem \ref{Thm:HC_classif}.

\section{Lusztig's quotient revisited}\label{S_Lusztig_quot}
\subsection{Special orbits and quantizations with integral central character}
Here we consider an important special case of the  conical symplectic singularities:
$Y:=\operatorname{Spec}(\C[\Orb])$, where $\Orb$ is a nilpotent orbit in a semisimple
Lie algebra $\g$. We will concentrate on the case when $\Orb$ is special. Let us recall
what this means.

Pick Cartan and Borel subalgebras $\mathfrak{h}\subset \mathfrak{b}\subset \g$.
Let $W$ denote the Weyl group of $\g$. Recall that the center of $U(\g)$
is identified with $\C[\h^*]^W$ via the Harish-Chandra isomorphism.
A two-sided ideal $\J\subset U(\g)$ is said to  have {\it integral central character}
if its intersection with the center is the maximal ideal in $\C[\h^*]^W$
of a point in the weight lattice $\Lambda$.  A nilpotent  orbit $\Orb$ is  called {\it special} if
there is a two-sided ideal $\J\subset U(\g)$ with integral central character
such that the associated variety of $U(\g)/\J$ is $\overline{\Orb}$.

Now let $\A_\lambda$ is a filtered quantization of $Y$. As such, it comes
equipped with a natural homomorphism $U(\g)\rightarrow \A_\lambda$, see,
e.g., \cite[Section 5]{orbit}. Let
$\J_\lambda$ denote the kernel. Clearly if $\J_\lambda$ has integral
central character, then $\Orb$ is special. Conversely, we have the following
result.

\begin{Prop}\label{Prop:spec_quant}
Let $\Orb$ be a special orbit. There is $\lambda\in \h_Y^*$ such that
$\J_\lambda$ has integral central character if and only if $\Orb$
is not one of the following four special orbits (in the Bala-Carter
notation):
\begin{itemize}
\item[(*)] $A_4+A_1$ in $E_7$, and $A_4+A_1, E_6(a_1)+A_1, A_4+2A_1$ in $E_8$.
\end{itemize}
\end{Prop}

This is \cite[Theorem 1.1]{quant_nilp}. In fact, one can find $\lambda$ explicitly,
see the next section.

%begin{proof}
%
%For exceptional Lie algebras the statement of the proposition was checked explicitly
%in \cite{quant_nilp}. As noted in that paper, the classical case can be done similarly,
%although no computation was provided. So let us explain how to prove the proposition
%for classical Lie algebras. As in \cite{quant_nilp}
%\end{proof}
%The statement for the classical Lie algebras was proved in \cite[Section 5]{W_dim}
%and the exceptional Lie algebra case was handled in \cite{quant_nilp}.
%Below we assume
%that $\Orb$ is special but not one of the four orbits in (*).

\subsection{Computation of $\lambda$}
Now we explain how to compute $\lambda\in \h^*_Y$. For this we first need
to recall results from \cite{orbit} on the computation of $\h^*_Y$ and $W_Y$ for
$Y=\operatorname{Spec}(\C[\Orb])$ for an arbitrary nilpotent orbit $\Orb\subset \g$.

It was proved in \cite[Theorem 4.4]{orbit} that $\Orb$ is {\it birationally induced} from a
{\it birationally rigid} orbit $\Orb'$ in a Levi subalgebra $\mathfrak{l}$ and the
pair $(\mathfrak{l},\Orb')$ is defined uniquely up to $G$-conjugacy. By definition,
this means the following. Let $\mathfrak{l}$ be a Levi subalgebra and $\mathfrak{p}=
\mathfrak{l}\ltimes \mathfrak{n}$ be a parabolic subalgebra with this Levi
subalgebra. Let $\Orb'$
be a nilpotent orbit in $\mathfrak{l}$. The subgroup $P\subset G$ acts
on $Y'\times \mathfrak{n}$, where we write $Y'$ for $\operatorname{Spec}(\C[\Orb'])$.
We have the generalized Springer morphism $G\times^P(Y'\times \mathfrak{n})\rightarrow \g$.
It is easy to see that its image is the closure of a single orbit, say $\Orb$, called {\it induced}
from $(\mathfrak{l}, \Orb')$. If the morphism  $G\times^P(Y'\times \mathfrak{n})\rightarrow \overline{\Orb}$
is birational, we say that $\Orb$ is {\it birationally induced} from $(\mathfrak{l},\Orb')$.
We say that $\Orb'$ is birationally rigid if it cannot be birationally induced from a proper Levi subalgebra.

The following result is a special case of  \cite[Proposition 4.6]{orbit}.

\begin{Lem}\label{Lem:h_W_for_orbits}
We have $\h^*_Y=\mathfrak{z}(\mathfrak{l})^*$ and $W_Y=N_G(\mathfrak{l},\Orb')/L$.
\end{Lem}

Now we state the main result of this section. Let $\Delta^{+,\mathfrak{n}}$ denote the set
of positive roots $\alpha$ such that the root space $\g_\alpha$ is in $\mathfrak{n}$.

\begin{Prop}\label{Prop:integral_period}
Suppose that $\Orb$ is special and is not one of the four orbits mentioned in
Proposition \ref{Prop:spec_quant}.
Set $\displaystyle \rho':=\frac{1}{2}\sum_{\alpha\in \Delta^{+,\mathfrak{n}}}\alpha$.
Then $\J_{\rho'}$ has integral central character.
\end{Prop}
%Note that this proposition strengthens Proposition \ref{Prop:spec_quant}.
\begin{proof}
It is known that $\Orb$ is induced from a birationally rigid special orbit
$\Orb'\subset \mathfrak{l}$, see  \cite[Proposition 2.3]{quant_nilp}. By our additional restriction on $\Orb$, the unique filtered quantization $\A'$ of $Y':=\C[\Orb']$ has integral central character, \cite[Theorem 1.1]{quant_nilp}.

Consider the homogeneous space $G/N$ and its sheaf of  differential operators, $D_{G/N}$. So we get a
sheaf of algebras $D_{G/N}\otimes \A'$ on $G/N$. The group $L$ acts on $D_{G/N}\otimes \A'$
in a Hamiltonian way, the quantum comoment map is the sum of the quantum comoment maps $\Phi_{D,\eta}, \Phi_{\mathcal{A}}$
for the actions of $L$ on $D_{G/N}$ and on $\A'$. Here  $\eta\in (\mathfrak{l}^*)^L$. These quantum comoment maps are as follows.  We set
$\Phi_{D,\eta}(x):=x_{G/N}-\langle \eta,x\rangle$,
and $\Phi_{\mathcal{A}}$ is chosen so that it vanishes of $\mathfrak{z}(\mathfrak{l})$.
We then can form the quantum Hamiltonian reduction of $D_{G/N}\otimes \A'$ with respect to the $L$-action
getting a sheaf of filtered algebras on $G/P$ to be denoted by $D^{\eta}(G,P,\A')$.
This sheaf can be viewed as a filtered quantization of $G\times^P(Y'\times \mathfrak{n})$.

We claim that the quantization $D^{-\rho'}(G,P,\A')$ has period zero. Thanks to
\cite[Theorem 5.4.1]{quant} this follows once we show that
\begin{itemize}
\item[(a)] the quantization
$D_{G/N}\otimes [\A'|_{(Y')^{reg}}]$ of $T^*(G/N)\times (Y')^{reg}$ is even (meaning that
it is the specialization at $h=1$ of a graded formal quantization that is even  in the sense of \cite[Section 2.3]{quant})
\item[(b)] and   the quantum comoment map
$\Phi_{D,-\rho'}+\Phi_\A: \mathfrak{l}\rightarrow D_{G/N}\otimes \A'$ is symmetrized
in the terminology of \cite[Section 5.4]{quant}.
\end{itemize}
To prove (a) we notice that both factors $D_{G/N}$ and $\A'|_{(Y')^{reg}}$ have period
zero (for the former this follows from \cite[Section 5]{quant}). A similar argument proves
(b).

So $H^0(G/P,\D^\eta(G,P,\A'))\cong \A_{\eta+\rho'}$. It remains to prove that
the kernel $\J_{\rho'}$ of $U(\g)\rightarrow H^0(G/P,\D^0(G,P,\A'))$ has integral central
character. Let $\mathfrak{m}'$ denote the intersection of the kernel of
$U(\mathfrak{l})\rightarrow \A'$ with the center of $U(\mathfrak{l})$. Then we
have an algebra homomorphism $\D^0(G,P, U(\mathfrak{l})/\mathfrak{m}')
\rightarrow \D^0(G,P,\A')$. The algebra
$U(\mathfrak{l})/\mathfrak{m}'$ is the algebra of twisted differential operators on
$P/B$ with twist, say $\mu$, that must be integral because the kernel of
$U(\mathfrak{l})\rightarrow \A'$ has integral central character. It is easy to
see that $H^0(G/P, \D^0(G,P, U(\mathfrak{l})/\mathfrak{m}'))\cong D^\mu(G/B)$.
So the homomorphism $\U\rightarrow H^0(G/P,\D^0(G,P,\A'))$ factors through
$\U\rightarrow D^\mu(G/B)$ and hence the kernel has integral central character.
\end{proof}

\subsection{Lusztig's quotient vs $\Gamma_\lambda$}
The following proposition describes the Lusztig quotient.

\begin{Prop}\label{Prop:Lusztig_quotient_comput}
Let $\Orb$ be as in Proposition \ref{Prop:integral_period}.
Let $\lambda\in \h_Y^*$ be such that $\J_\lambda$ has integral central character.
Then the Lusztig quotient $\bA_c$ coincides with $\Gamma/\Gamma_\lambda$.
\end{Prop}
\begin{proof}
The proof is in several steps. Let us start by relating the construction of the present paper to that of \cite{LO}.

{\it Step 1}.
We note that the ideal $\J_\lambda$ is completely prime and has central character, so is primitive.
Consider the W-algebra $\Walg$ constructed from the orbit $\Orb$.
In \cite{HC} (see, for example, Theorem 1.2.2 there)
the author produced an ${\bf A}(\Orb)$-orbit of finite dimensional irreducible
representations of $\Walg$ starting from a primitive ideal in $U(\g)$.
In our case, $U(\g)/\J_\lambda$ has multiplicity $1$ on $\Orb$. It follows that the orbit consists
of a single 1-dimensional representation.

{\it Step 2}. Consider the category $\operatorname{HC}(U(\g)/\J_\lambda)$ of the HC $U(\g)$-bimodules annihilated by $\J_\lambda$ on the left and on the right and its full subcategory $\operatorname{HC}^{\partial \Orb}(U(\g)/\J_\lambda)$ of bimodules supported on $\partial\Orb$. Let $\HC_{\Orb}(U(\g)/\J_\lambda)$ denote the quotient category.  The inclusion $U(\g)/\J_\lambda\hookrightarrow \A_\lambda$ gives rise to
the forgetful functor $\HC(\A_\lambda)\rightarrow \HC(U(\g)/\J_\lambda)$,
which, in turn, induces a functor $\overline{\HC}(\A_\lambda)\rightarrow
\HC_{\Orb}(U(\g)/\J_\lambda)$. We claim that this functor is an equivalence.
For this we consider the functor $\bullet_{\dagger}$ from \cite[Section 3.4]{HC}.
By \cite[Theorem 1.3.1]{HC}, this functor identifies $\HC_\Orb(U(\g)/\J_\lambda)$ with a full subcategory of the category $\operatorname{Bim}^Q(\Walg/\I_\lambda)$ of $Q$-equivariant  $\Walg/\I_\lambda$-bimodules. Here the notation is as
follows. By $\I_\lambda$ we denote the image of $\J_\lambda$ under $\bullet_{\dagger}$, this is a two-sided ideal of codimension
$1$. Finally, $Q$ is the reductive part of the centralizer $Z_G(e)$ of $e\in \Orb$, where $G$
is the adjoint group with Lie algebra $\g$. The group $Q$ acts on $\Walg$ in a Hamiltonian
way. The functor $\bullet_{\dagger}:
\operatorname{HC}(U(\g)/J_\lambda)\rightarrow \operatorname{Bim}^Q(\Walg/\I_\lambda)$
has a right adjoint functor $\bullet^{\dagger}$ that is also a left inverse for
the functor $\operatorname{HC}_\Orb(U(\g)/J_\lambda)\rightarrow \operatorname{Bim}^Q(\Walg/\I_\lambda)$.
It was checked in \cite[Lemma 5.2]{orbit} that $\A_\lambda=(\Walg/\I_\lambda)^{\dagger}$.
This implies that $\overline{\HC}(\A_\lambda)\xrightarrow{\sim} \HC_{\Orb}(U(\g)/\J_\lambda)$.

{\it Step 3}.
Note that $Q/Q^\circ=\Gamma/\pi_1(G)$. Below we
will write $\underline{\Gamma}$ for $Q/Q^\circ$. Clearly, $\operatorname{Bim}^Q(\Walg/\I_\lambda)$
is identified with $\C\underline{\Gamma}\operatorname{-mod}$. Comparing the constructions
of $\bullet_{\dagger}$ in \cite[Section 3.4]{HC} and in the present paper we see that the embedding
$\overline{\HC}(\A_\lambda)\hookrightarrow \C\underline{\Gamma}\operatorname{-mod}$ described above
in this step coincides with what we have constructed in Section \ref{SS_HC_equiv}.
In particular, $\Gamma_\lambda$ contains $\pi_1(G)$. Let us write $\underline{\Gamma}_\lambda$
for $\Gamma_\lambda/\pi_1(G)$.

{\it Step 4}. Our goal is to show that  $\underline{\Gamma}_\lambda$ coincides with the kernel of
$\underline{\Gamma}\twoheadrightarrow \bA_c$. First of all, note that
$|\bA_c|=|\underline{\Gamma}/\underline{\Gamma}_\lambda|$. Indeed, both numbers are equal to the number of
simples in $\overline{\HC}(\A_\lambda)$: for the left hand side this follows
from \cite[Theorem 1.1]{LO}, while for the right hand side this is a consequence of
Theorem \ref{Thm:HC_classif}. Now the case when $\bA_c$ is not abelian
is easy: in all such cases this group coincides with $\underline{\Gamma}$.

{\it Step 5}.
Now let us assume that $\bA_c$ is abelian. This is always the case when $\g$ is classical. If $\g$ is exceptional
and $\bA_c$ is abelian, then  $\bA_c\cong \Z/2\Z$, see, e.g.
\cite[Section 6.7]{LO}. It was shown in
\cite[Sections 6.5-6.7]{LO} that there is a finite dimensional irreducible representation of
the central reduction $\Walg_\rho$
whose stabilizer in $\underline{\Gamma}$ is the kernel  of
$\underline{\Gamma}\twoheadrightarrow \bA_c$, to be denoted by
$\underline{\Gamma}^0$.
Let $\J'_\rho$ be the corresponding primitive ideal in $U(\g)$.
It follows from  \cite[Sections 7.4, 7.5]{LO} that the quotient of the category of HC $U(\g)/\J'_\rho$-$U(\g)/\J_\lambda$-bimodules
modulo the bimodules supported on $\partial\Orb$ is equivalent to $\operatorname{Vect}$.
The image of this quotient category in the category of semisimple
finite dimensional $Q$-equivariant $\Walg$-bimodules
is closed under tensoring by the images under $\bullet_{\dagger}$ of objects in
$\HC(\U_\lambda/\J_\lambda)$. These images are precisely the representations of $\underline{\Gamma}/\underline{\Gamma}_\lambda$ by Theorem \ref{Thm:classif_sympl_prelim}. As a right
module category over $\C\underline{\Gamma}\operatorname{-mod}$, the quotient category  has
the form $\operatorname{Rep}^\psi\underline{\Gamma}^0$, where $\psi$ is some Schur multiplier.
It follows from \cite[Remark 7.7]{LO} and the existence of an ${\bf A}(\Orb)$-stable
one-dimensional representation of $\Walg$ with integral central character (that follows
from \cite[Theorem 1.1]{quant_nilp}) that $\psi$ is a coboundary.
So we have an irreducible representation $V\in \operatorname{Rep}(\underline{\Gamma}^0)$
with the following property: $V\otimes U\cong_{\underline{\Gamma}^0} V^{\oplus \dim U}$ for any
representation $U$ of $\underline{\Gamma}/\underline{\Gamma}_\lambda$.

{\it Step 6}. If $\bA_c$ is abelian, then so is  $\underline{\Gamma}$.
This can be seen from explicit computations of $\bA_c$, see, e.g., \cite[Section 6.7]{LO},
 and  the tables from
\cite[Section 13.3]{Carter}. All irreducible representations of $\underline{\Gamma}^0$
are 1-dimensional. So the condition $V\otimes U\cong_{\underline{\Gamma}^0} V^{\oplus \dim U}$ implies that
$U$ must be trivial over $\underline{\Gamma}^0$ and hence $\underline{\Gamma}^0\subset \underline{\Gamma}_\lambda$.
But the cardinalities of these two groups are the same by Step 4
and so we get $\underline{\Gamma}_0=\underline{\Gamma}_\lambda$.
%{\it Step 7}. Now let $\g$ be exceptional.
%Then $\underline{\Gamma}^0$ is an index 2 subgroup of $\underline{\Gamma}$ and so is
%$\underline{\Gamma}_\lambda$. In all cases the group ${\bf A}$ has only one quotient of
%order $2$, as can be deduced from the tables, say, in \cite[Section 8.4]{CM}. So, again, %$\underline{\Gamma}^0=\underline{\Gamma}_{\lambda}$.
%It follows that the nontrivial irreducible $\Gamma/\Gamma_\lambda$-module
%is 1-dimensional. From $V\otimes U\cong V$
%
%It was shown in \cite[Remark 7.7]{LO} that each $\Gamma/\pi_1(G)$-module in the image
%of $\overline{\HC}(\A_\lambda)$ is a projective representation of $\bA_c$ with
%coboundary Schur multiplier.
\end{proof}

\begin{Rem}\label{Rem:bar_A_induction}
In \cite{orange} Lusztig observed that $\bA_c$ coincides with $\mathbf{A}(\Orb_1)$
for a suitable nilpotent orbit $\Orb_1$ in a Levi subalgebra $\mathfrak{l}_1$.
His proof is a case by case argument. We would like to sketch how this result
follows from Proposition \ref{Prop:Lusztig_quotient_comput}.

First of all, note that if $\Orb$ is birationally induced from $(\mathfrak{l}_1,\Orb_1)$,
then $\Orb\hookrightarrow G\times^{P_1}(\Orb_1\times \mathfrak{n}_1)$, which gives an
epimorphism $\pi_1(\Orb)\twoheadrightarrow \pi_1(G\times^{P_1}(\Orb_1\times \mathfrak{n}_1))
\cong \pi_1(\Orb_1)$.

Identify $\mathfrak{z}(\mathfrak{l})$ with $\mathfrak{z}(\mathfrak{l})^*$ via the
Killing form of $\g$.  Now pick a small complex neighborhood $U$ of $\lambda\in \h^*_Y=
\mathfrak{z}(\mathfrak{l})^*$. The locus of $\lambda'\in U$ with $\Gamma_{\lambda'}=\Gamma_\lambda$
has the form $(\lambda+\Pi)\cap U$ for a uniquely determined vector subspace $\Pi\subset
\mathfrak{z}(\mathfrak{l})^*$. Take the centralizer of $\Pi\subset \mathfrak{z}(\mathfrak{l})$
in $\g$ for $\mathfrak{l}_1$. For $\Orb_1$ we take the orbit in $\mathfrak{l}_1$ induced
from $(\mathfrak{l}',\Orb')$. Then  one can  see that
${\bf A}(\Orb_1)=\underline{\Gamma}/\underline{\Gamma}_\lambda$.
\end{Rem}

\begin{Rem}\label{Rem:quot_remaining_cases}
Let us explain how to handle the remaining four orbits. Three of them: $A_4+A_1$ in $E_7,E_8$
and $A_4+2A_1$ have codimension of the boundary $\geqslant 4$. In this case, \cite[Proposition 4.7]{quant_nilp}
implies that $\bA={\bf A}(\Orb)$ (all these orbits are birationally rigid so \cite[Proposition 4.7]{quant_nilp}
formally applies but, if fact, the proof only uses the condition on the codimension of the boundary).
In all these cases, ${\bf A}(\Orb)$ is $\Z/2\Z$.

Let us explain how to handle the remaining case, $E_6(a_1)+A_1$ in $E_8$. Here ${\bf A}(\Orb)\cong \Z/2\Z$
as well. This orbit is birationally induced from $A_4+A_1$ in $E_7$. Thanks to \cite[Theorem 1.1]{LO}, the claim that
$\bA={\bf A}(\Orb)$ is equivalent to the existence of a 1-dimensional representation of $\Walg$ with integral central character that is not ${\bf A}(\Orb)$-stable. Such a representation, say $\underline{N}$,
exists for the orbit $\underline{\Orb}$ of type $A_4+A_1$ in $E_7$
because $A_4+A_1$ is Richardson. 
Take such  a representation and extend it to a representation of the W-algebra for the
corresponding Levi subalgebra, which amounts in specifying the character, say $\chi$, of the action of
the center of the Levi. Then we can induce the resulting 1-dimensional representation, see
\cite[Section 6]{W_1dim}, to get a representation, say $N_\chi$, of the W-algebra $\Walg$ for $\Orb$.
When $\chi$ is integral,   $N_\chi$ has integral central character. On the other hand,
the set of all $\chi$ such that $N_\chi$ is ${\bf A}(\Orb)$-stable is Zariski closed.
It is not difficult, but somewhat technical, to show that if $N_{\chi}$ is ${\bf A}(\Orb)$-stable
for all $\chi$, then $\underline{N}$ is ${\bf A}(\underline{\Orb})$-stable. We arrive at a contradiction
that shows that $N_{\chi}$ is not ${\bf A}(\Orb)$-stable for some integral $\chi$.
\end{Rem}

%We also have the following result.
%
%\begin{Lem}\label{Lem:spec_induction}
%The following claims hold.
%\begin{enumerate}
%\item  Every special orbit is induced from a birationally rigid special orbit.
%\item  Assume $\Orb$ is not one of the four orbit mentioned above. Then
%the unique filtered quantization $\A'$ of $\C[\Orb']$ has integral central
%character.
%\end{enumerate}
%\end{Lem}

\end{document}